\documentclass{amsart}
\usepackage{setspace}
\usepackage[T1]{fontenc}
\usepackage[utf8]{inputenc}

\usepackage[hyperfootnotes,colorlinks=true,citecolor=cyan,backref=page]{hyperref}

\usepackage[normalem]{ulem}

\usepackage{lmodern}

\usepackage{url}
\usepackage{graphicx}
\usepackage{amsthm} 
\usepackage{amsfonts}
\usepackage{amsmath} 
\usepackage{amssymb}
\usepackage{latexsym}
\usepackage{pdfpages}
\usepackage{subfig}
\usepackage{tikz-cd}
\usepackage{mdframed}
\usepackage{setspace}
\usepackage{comment}
\usepackage{cancel}
\usepackage{caption}
\usepackage{subcaption}

\usepackage{tikz}
\usetikzlibrary{calc,shapes,intersections,patterns,arrows}

\usepackage{indentfirst}

\newtheorem{Theorem}{Theorem}[section]
\newtheorem{Lemma}[Theorem]{Lemma}
\newtheorem{Proposition}[Theorem]{Proposition}
\newtheorem{Corollary}[Theorem]{Corollary}

\newtheorem{Problem}{Problem}
\theoremstyle{definition}
\newtheorem{Definition}[Theorem]{Definition}
\newtheorem{Remark}[Theorem]{Remark}

\newtheorem{Example}[Theorem]{Example}

\newcommand{\dom}{{\operatorname{Dom}}}
\newcommand{\Red}{{\operatorname{Red}}}
\newcommand{\dep}{\operatorname{dp}}

\newcommand{\cone}{\operatorname{cone}}
\newcommand{\ord}{\operatorname{ord}}

\newcommand{\RPal}{\operatorname{Pal}}
\newcommand{\Pref}{\operatorname{Pref}}

\newcommand{\mo}{\mathcal{O}}

\newcommand{\pal}{\operatorname{pal}}

\title[Generating functions and automata associated to reflections]{On generating functions and automata associated to reflections  in Coxeter systems}

\author[R. Biagioli]{Riccardo~Biagioli}
\address[Riccardo Biagioli]{Alma Mater Studiorum – Università di Bologna\\ Piazza di Porta San Donato 5, 40126 Bologna\\Italy}
\email{riccardo.biagioli2@unibo.it}

\author[C. Hohlweg]{Christophe~Hohlweg}
\address[Christophe Hohlweg]{Universit\'e du Qu\'ebec \`a Montr\'eal\\
LaCIM et D\'epartement de Math\'ematiques\\ CP 8888 Succ. Centre-Ville\\
Montr\'eal, Qu\'ebec, H3C 3P8\\ Canada}
\email{hohlweg.christophe@uqam.ca}
\urladdr{http://hohlweg.math.uqam.ca}

\author[E. Sasso]{Elisa Sasso}
\address[Elisa Sasso]{Alma Mater Studiorum – Università di Bologna\\ Piazza di Porta San Donato 5, 40126 Bologna\\Italy}
\email{elisa.sasso2@unibo.it}

\thanks{Hohlweg was partially supported by the NSERC grant {\em Combinatorics of infinite Coxeter groups} held by Hohlweg, and Biagioli by the PRIN 2022 Project 2022S8SSW2 {\em Algebraic and geometric aspects of Lie theory}.}

\begin{document}

\maketitle
\begin{abstract}
In this article, we study two combinatorial problems concerning the set of reflections of a Coxeter system.
The first problem asks whether the language of palindromic reduced words for reflections is regular, and the second is about finding formulas for the Poincaré series of reflections, namely the generating function of reflection lengths. These two problems were inspired by a conjecture of Stembridge stating that the Poincaré series of reflections is rational and by the solution provided by de Man. 

To address the first problem, we introduce reflection-prefixes, arising naturally from palindromic reduced words.
We study their connections with the root poset, dominance order on roots, and dihedral reflection subgroups.
Using  $m$-canonical automata associated with 
$m$-Shi arrangements, we prove that the language 
  of reduced words for reflection-prefixes is regular.
For the second problem, we focus on affine Coxeter groups.
In this case, we derive a simple formula for the Poincaré series using symmetries of the Hasse diagram of the root poset. 
\end{abstract}

\section{Introduction} 

Let $(W,S)$ be a Coxeter system, with $S$ finite. We denote by $S^*$ the free monoid on the alphabet $S$. To distinguish between a word and its corresponding group element, we denote a word in $S^*$ using bold letters, $\bold w =  s_1 \cdots  s_k$, while the resulting product in $W$ is denoted by $w=s_1\cdots s_k$. Let $w\in W$, a word $\bold w= s_1 \cdots  s_k\in S^*$ is a {\em reduced word for $w$} if  $w=s_1\cdots s_k$  and $k$ is minimal for this property; in this case the {\em length} of $w$ is $\ell(w)=k$. The identity $e\in W$ is represented by the empty word, and $\ell(e)=0$. We denote the set of all reduced words of $w\in W$ by $\Red(w)$, and the set of all reduced words in $W$ by $\Red=\bigsqcup_{w\in W} \Red(w)$, where $\sqcup$ denotes the disjoint union.
Finally, for any subset $A\subseteq W$, the {\em Poincar\'e series of $A$} is the formal power series
$A(q):=\sum_{w\in A} q^{\ell(w)}.$
\smallskip

In this article, we study two combinatorial problems regarding the {\em set of reflections} $T:=\bigcup_{w\in W} wSw^{-1}$ of a Coxeter group $W$:
\begin{enumerate}
    \item Is the language of {\em palindromic reduced words (for the reflections)} regular?
    \item Are there explicit and elegant formulas for the Poincar\'e series of the set $T$ of reflections ?
\end{enumerate}
Our main contributions to those questions are the following: 
\begin{itemize}
\item We introduce the notion of  {\em reflection-prefixes}, a class of elements in $W$ arising naturally from  palindromic reduced words of reflections, and study their properties in relation to the {\em root poset}, the {\em dominance order on roots}  and {\em dihedral reflection subgroups}.
\item For any Coxeter system, we show that the {\em language of reduced words for reflection-prefixes}, $\Pref_T$, is regular. This is achieved using the family of $m$-canonical automata associated with $m$-Shi arrangements (see~\cite[\S3.4]{HNW} and \cite{DFHM24}), which provide a family of finite deterministic automata recognizing $\Pref_T$. As a consequence, we show that the generating function  of palindromic reduced words is rational.
\item In the case of affine Coxeter groups, we derive a simple expression for the Poincar\'e series of the set of reflections in terms of rational fractions. This formula is obtained from some symmetries of the Hasse diagram of the root poset.
\end{itemize}

We discuss now some history and motivations that lead to this article. 
If $W$ is finite, $A(q)$ is clearly a polynomial for any $A\subseteq W$. A natural question in combinatorics of infinite Coxeter systems is to classify for which subsets $A$ the Poincar\'e series $A(q)$ is {\em rational}, that is, can be written as a ratio of two polynomials in $q$. It is well-known that $W(q)$ is rational, see for instance~\cite[Corollary~7.1.8]{BB}. Furthermore, an explicit recursive formula for $W(q)$ in term of standard parabolic subgroups is provided in~\cite[Proposition~7.1.7]{BB}. 

A further natural direction is the study of generating functions of words in $S^*$ in relation to the Coxeter system $(W,S)$.  More precisely, given the canonical projection $\pi:S^* \rightarrow W$ sending a word $\bold w= s_1 \cdots  s_k$ to the element $w=s_1\cdots s_k$,
it is natural to  consider the {\em ‘‘lifted'' Poincar\'e series} of a subset $B\subseteq S^*$ relative to $(W,S)$:
$$
B(q)=\sum_{\bold w\in B} q^{\ell(w)},
$$
where the notation ${\ell(w)}$ is understood to be $\ell(\pi({\bf w}))$. 

A longstanding open problem in combinatorics of Coxeter groups is the enumeration of $|\Red(w)|$; for a detailed discussion see~\cite[p.123]{BB}, and for various partial results, see~\cite{Eriksson,Hart}.  However, if we consider the set of all elements of a given length, we know that the numbers $r_k$ of reduced words of length $k\in \mathbb N$ is enumerated by  the following Poincar\'e series
$$
\Red(q)=\sum_{\bold w\in \Red} q^{\ell(w)}=\sum_{k\in\mathbb N} r_kq^k,
$$
which is known to be rational. The proof follows from the existence of a finite deterministic automaton that recognizes the language $\Red$; see for instance~\cite[Theorem~4.9.1]{BB} for more details. 

\smallskip

As reported by Brenti~\cite{B1}, Stembridge proposed the following problem during an open problem session at the Mathematical Sciences Research Institute at Berkeley in 1997: Is it true that the Poincar\'e series
$$
T(q)=\sum_{t\in T} q^{\ell(t)}\quad \textrm{is rational?}
$$ 

\textbf{}Since reflections are known to admit palindromic reduced words, a natural approach to this question is to construct a finite automaton that recognizes exactly one palindromic reduced word for each reflection. The existence of such a regular language would imply that the associated generating function is rational. However, the property of being regular is not generally preserved under the palindromic constraint. If a language is regular, it is well-known that the sublanguage of its palindromic words is context-free but not necessarily regular; for instance if $S=\{a,b\}$ then $S^*$ is regular but a standard application of the {\em pumping lemma} shows that the sublanguage of its palindroms is not.  Indeed, the regularity of $\RPal$, {\em the language of all palindromic reduced words}, which are necessarily reduced words for reflections, remains an open question. In a recent article, Mili\'cevi\'c ~\cite{Mili} provides palindromic reduced words for all reflections in finite Weyl groups, but highlights the challenge of finding a general algorithm applicable to any Coxeter group. To overcome these challenges, we shift to the language of reflection-prefixes which we define below.

\subsection*{The language of reflection-prefixes}  

We recall that the {\em (right) weak order} $(W,\leq_R)$ is the poset defined by $u\leq_R w$ if there exists a reduced word for $u\in W$ that is a prefix for a reduced word for $w\in W$. Equivalently, $u\leq_R w$ if and only if $\ell(w)=\ell(u^{-1}w)+\ell(u)$.

It is well-known that every reflection has an odd length. Let $t\in T$ be a reflection of length $\ell(t)=2k+1$, we define a {\em $t$-prefix} to be any element $p_t\in W$ such that $p_t\leq_R t$ and $\ell(p_t)=k+1$.
  It turns out that $ s_1\cdots s_k s_{k+1}  s_{k}\cdots  s_1$ is a palindromic reduced word for $t$  if and only if $s_1\cdots  s_k s_{k+1}$ is a reduced word for a $t$-prefix; see \S\ref{ss:reflectionprefixes}. 

Reflection-prefixes turn out to be a useful combinatorial framework to study the root poset (Proposition~\ref{prop:PrefRoot}), the dominance order (Proposition~\ref{prop:PrefDom}) and to find the canonical generators of dihedral reflection subgroups (Theorem~\ref{prop:algo}).  Our main results concerning reflection-prefixes and their associated languages are summarized in the following theorem.

\begin{Theorem}\label{thm:PrefMain} Let $(W,S)$ be a Coxeter system. 
\begin{enumerate}
\item The language $\Pref_T$ of reduced words for reflection-prefixes is regular.
\item The language of palindromic reduced words is:
$$
\RPal=\{s_1\cdots  s_k s_{k+1}  s_{k}\cdots  s_1 \in S^*\mid  s_1\cdots  s_k s_{k+1}\in\Pref_T\}.
$$
\item The following Poincar\'e series is rational:
$$
\Pref_T(q)=\sum_{\bold w \in \Pref_T} q^{\ell(w)}=\sum_{k\geq 0} \pal_kq^k,
$$
where $\pal_k$ is the number of palindromic reduced words of length $k$.

\item The Poincar\'e series $\RPal(q)=q\Pref_T(q^2)$ is rational.
\end{enumerate} 
\end{Theorem}

The first assertion, the regularity of $\Pref_T$, is established in Corollary~\ref{cor:Refregular}. The proof relies on a construction of automata derived from the family of Garside shadow automata associated with $m$-Shi arrangements $(m \in \mathbb{N})$; for further details, see~\cite{HNW, DFHM24}. The second assertion follows from Proposition~\ref{prop:Pref1}, while the third and fourth are consequence of the first two.

\subsection*{Closed formulas for $T(q)$ in affine Coxeter systems} Remarkably, a solution to Stembridge's problem was provided in 1999 by de~Man~\cite{DM}. However, the author seemed unaware of Stembridge's question and the combinatorics community was not aware of the solution. This oversight may be partly due to the title of de~Man's article {\em ``The generating function for the number of roots of a Coxeter group''}, which only address counting reflections by length at the end of the article in \cite[\S5.2]{DM}. His solution is based on the Brink--Howlett automaton~\cite{BH}  that recognizes the language of {\em lexicographically ordered reduced word}, that is, each element $w$ has a unique lexicographically reduced word in $\Red(w)$.  

More precisely, in \cite[Theorem 4.1]{DM}, de Man shows that the generating function  $\Phi^+(q)$ of the depth of positive roots is a sum of rational functions that are obtained from a subset of states of the Brink-Howlett automaton and from the enumeration of the vertices of a family of graphs parameterized by $T$. So the generating function $\Phi^+(q)$ is rational. Since $T(q)=q\, \Phi^+(q^2)$, the Poincar\'e series $T(q)$  is rational as well. However, the Brink--Howlett automaton contains a lot of states, most of them ``dead states'', which make this automaton difficult to use in practice. For instance, in affine type $\widetilde{A}_2$, the number of states of the Brink--Howlett automaton is bounded by $2^6$; de Man provides some example of $T(q)$ using a computer to handle his formula.  
 
 In this article, we give, in the case of affine Coxeter systems, an alternative solution to Stembridge's problem by providing a simple formula for  $T(q)$. 

\begin{Theorem}\label{thm:Main1} Let $(W,S)$ be an affine Coxeter system. Then 
$$
T(q):=\sum_{t\in T}q^{\ell(t)}=q\frac{P(q^2)}{1-q^{2M}},
$$ 
where $M\in \mathbb{N}$ and $P(q)$ is a palindromic polynomial.
\end{Theorem}

The proof of the above theorem relies on two fundamental symmetries of the root poset on $\Phi^+$ (Theorems \ref{th:iso1} and \ref{th:iso2}). These properties yield an explicit formula for $M$ and $P(q)$ (Theorem \ref{th:main})), from which the above result follows.

\subsection*{Plan of the article} The article is organized as follows. In~\S\ref{sec:root-poset}, we recall some needed facts about the combinatorics of words and roots and, in particular, about two partial orders on the root system of a Coxeter system: the root poset and the dominance order, which were introduced by Brink and Howlett~\cite{BH} in their work to show that Coxeter systems are automatic. In~\S\ref{se:reflectionprefix}, we discuss reflection-prefix, palindromic reduced words for reflections and automata and, in \S\ref{se:affine},  we give the enumeration of reflections by length in affine types. Finally, we propose some open problems in \S\ref{se:openproblems}.

\subsection*{Acknowledgements}  CH warmly thanks Francesco Brenti, James Parkinson and Christophe Reutenauer for very instructive discussions. 
The authors are particularly indebted to Matthew Dyer for bringing de Man’s article~\cite{DM} to their attention during the early stages of this project.

\section{Combinatorics of words and roots in Coxeter systems}\label{sec:root-poset}

In this section, we review several standard results regarding the combinatorial theory of Coxeter groups; we refer the reader to \cite{BB} for a more detailed exposition.

\subsection{Combinatorics of words in Coxeter systems} Let $(W,S)$ be a Coxeter system. As recalled in the introduction, if the word ${\bf w}=s_1 \cdots  s_k$ in $S^*$ is a reduced word for $w\in W$, we simply say that the product  $w=s_1\cdots s_k$ is a reduced word for $w\in W$. Similarly, for $u, v, w \in W$, we say that $w = uv$ is \textit{a reduced product} if $\ell(w) = \ell(u) + \ell(v)$. In this case,  we say that $u$ is a \textit{prefix} of $w$ and $v$ is a \textit{suffix} of $w$. In other words, $u$ is a prefix (resp. suffix) of $w$ if and only if there exists a reduced word for $u$ that is a prefix (resp. suffix) of some reduced word for $w$. 

\smallskip

Reduced words of $w\in W$ are in bijections with geodesics starting at $e$ and ending at $w$ in the (right) Cayley graph of $(W,S)$. Recall from the introduction that the \textit{(right) weak order} $(W,\leq_R)$ is the poset defined by: $u \leq_R w$ if $u$ is a prefix of $w$. The (right) weak order has for Hasse diagram the (right) Cayley graph for which the edges $\{w,ws\}$ are oriented from $w$ to $ws$ if $\ell(w)<\ell(ws)$. 

The {\em right descent set of $w\in W$} is 
$$
D_R(w)=\{s\in S\mid \ell(ws)<\ell(w)\}=\{s\in S\mid ws\leq_R w\}.
$$ 
Similarly, the {\em left descent set of $w\in W$} is 
$$
D_L(w)=\{s\in S\mid \ell(sw)<\ell(w)\}=\{sw\in S\mid s\leq_R w\}.
$$
For $I\subseteq S$, we consider {\em standard parabolic subgroup $W_I=\langle I\rangle$}. The set of minimal coset representatives of $W/W_I$ is 
$$
X_I=\{x\in W\mid x\leq_R xs,\ \forall s\in I\}.
$$
In particular, any element $w$ has a unique reduced product $w=uv$ with $u\in X_I$ and $v\in W_I$. Moreover, any product $uv$ in $W$ where  $u\in X_I$ and $v\in W_I$ is a reduced product; see for instance~\cite[Proposition~2.4.4]{BB}. 
The following statement is well-known; we include a proof for completeness. 

\begin{Lemma}\label{lem:Descents} Let $w\in W$ and $I\subseteq D_R(w)$. Then $W_I$ is finite and  $w=uw_{\circ,I}$ is a reduced product, where $u\in X_I$ and $w_{\circ,I}$ is the longest element in $W_I$. 
\end{Lemma}
\begin{proof} We consider the reduced product $w=uv$, where $u\in X_I$ and $v\in W_I$. Since $I\subseteq D_R(w)$ and by definition of reduced product,  we have for all $s\in I$:
$$
\ell(w)-1=\ell(ws)=\ell(u(vs))=\ell(u)+\ell(vs).
$$
So $\ell(vs)<\ell(v)$ since $\ell(w)=\ell(u)+\ell(v)$. In other words, $D_R(v)=I$, since $v\in W_I$. Therefore, by~\cite[Proposition~2.3.1~(ii)]{BB} applied to $W_I$, we obtain that $W_I$ is finite and $v=w_{\circ,I}$.
\end{proof}

Bj\"orner showed that $(W,\leq_R)$ is a meet-semilattice, meaning that any nonempty subset $X\subseteq W$ admits a {\em meet}, that is a greatest lower bound denoted by $\bigwedge_R X$; see~\cite[\S3.2]{BB}. A subset $X \subseteq W$ is \textit{bounded} in $W$ if there exists $g \in W$ such that $x \leq_R g$ for any $x \in X$. Therefore, any bounded subset $X \subseteq W$ admits a least upper bound $\bigvee_R X$ called the \textit{join} of $X$: 

$$
\bigvee_R X = \bigwedge_R\{g \in W \mid x \leq_R g, \, \forall x \in X\}\in W.
$$

\subsection{Garside shadows automata}\label{ss:Garside}

A \textit{Garside shadow for $(W,S)$} is a subset $G\subseteq W$ such that $S\subseteq G$ and
\begin{itemize}
\item[(i)] $G$ is closed under join in the right weak order, i.e. if $\emptyset \not = X\subseteq G$ is bounded, then $\bigvee_R X \in G$; 
\item[(ii)]$G$ is closed under suffixes, i.e. if $w\in G$ and $w=uv$ is reduced, then $v\in G$.
\end{itemize}
Associated with any Garside shadow $G$ is a {\em Garside projection} defined as:
\begin{eqnarray*}
\pi_G:W&\mapsto& G \\
w&\mapsto& \bigvee_R \{g\in G\mid g\leq_R w\}.
\end{eqnarray*}
Garside shadows were introduced by Dyer and the second author to show the existence of a finite Garside family in Artin-Tits monoids, see \cite{DH}. In particular, they show that finite Garside shadows exist  whenever $S$ is finite. 

\smallskip
In~\cite{HNW}, the second author together with Nadeau and Williams defined for any Garside shadow a deterministic automaton that recognized $\Red$, the language of reduced words for $(W,S)$.  Let $G$ be a Garside shadow for $(W,S)$. The  {\em Garside shadow automaton $\mathcal A_G$} is the deterministic automaton over the alphabet $S$ defined as follows:
\begin{itemize}
\item the set of {\em states} of $\mathcal A_G$ is $G$;
\item the {\em final states} are the elements in $G$;
\item the {\em initial state} is the identity $e$;
\item the {\em transitions} are: $w\in W \xrightarrow{s} \pi_G(sw)\in G$, whenever $\ell(sw)>\ell(w)$.
\end{itemize}
The automaton $\mathcal A_G$ is represented by a directed edge-labeled graph on the vertex set $G$ with edges labeled by elements of $S$, such that for any $w \in G$ and $s \in S$, there is at most one edge (a transition above) with source $w$ and label $s$.

\begin{Theorem}[{\cite[Theorem~1.2]{HNW}}]\label{thm:GarsideAutomaton} Let $G$ be a finite Garside shadow for $(W,S)$, then $\mathcal A_G$ is a finite deterministic automaton that  recognizes the language $\Red$.
\end{Theorem}

As a consequence, the authors recover that $\Red$ is regular and that the Poincar\'e series $\Red(q)$ is rational. We refer the reader also to~\cite{PY22,PY24,Sa25} for recent developments around Garside shadow automata. 

\begin{Example}\label{ex:aut1}
Let $(W,S)$ be a Coxeter system of type 
\begin{center}
\begin{tikzpicture}
	[scale=2,
	 q/.style={teal,line join=round},
	 racine/.style={blue},
	 racinesimple/.style={blue},
	 racinedih/.style={blue},
	 sommet/.style={inner sep=2pt,circle,draw=black,fill=blue!40,thick,anchor=base},
	 rotate=0]
 \tikzstyle{every node}=[font=\small]
\def\grosseursimple{0.025}
\coordinate (ancre) at (0,3);

\node[sommet,label=above:$1$] (a2) at ($(ancre)+(0.25,0.4)$) {};
\node[sommet,label=below right :$3$] (a3) at ($(ancre)+(0.5,0)$) {} edge[thick] node[auto,swap,right] {}(a2) ;
\node[sommet,label=below left:$2$] (a4) at (ancre) {} edge[thick] node[auto,swap,below] {$\infty$} (a3) edge[thick] node[auto,swap,left] {} (a2);
\end{tikzpicture}
\end{center}

 The set $G=W_{\{s_1,s_2\}}\cup W_{\{s_1,s_3\}}$ is a Garside shadow for $(W,S)$ and its associated Garside shadow automaton $\mathcal{A}_{G}$ is the one represented in Figure \ref{fig:ex aut1}. The transitions labeled by $s_1$ are colored orange, $s_2$ in purple, and $s_3$ in green.

\begin{figure}[h]
     \centering
     \includegraphics[width=0.5\linewidth]{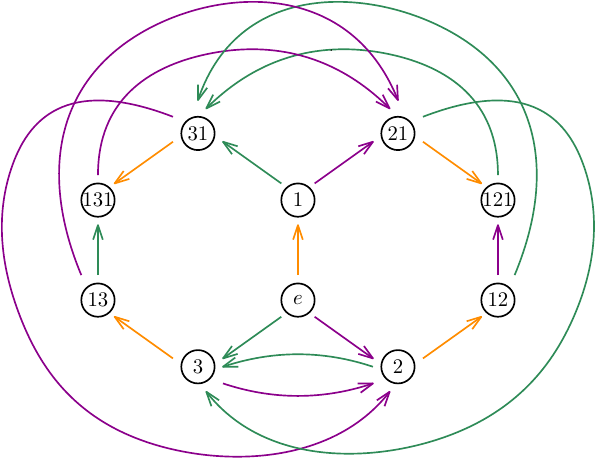}
     \caption{The automaton $\mathcal{A}_{G}$ of Example \ref{ex:aut1}.}
     \label{fig:ex aut1}
\end{figure}

\end{Example}

\subsection{Root systems and inversion sets}\label{ss:InversionSets}
We recall some basic facts about root systems and inversion sets; for more details, see~\cite[\S2.3]{DH}. 
\smallskip

Recall that a quadratic space $(V, B)$ is the data of a finite dimensional real vector space $V$ endowed with a symmetric bilinear form $B$. We denote by $O_B(V)$ the group of linear maps of $V$ preserving $B$ and by $Q = \{v \in V \mid B(v,v) = 0\}$ the {\em isotropic cone} of $(V,B)$.

Any Coxeter system $(W,S)$ admits a {\em geometric representation}, that is, a faithful representation of $W$ as a discrete reflection subgroup of $O_B(V)$ for some quadratic space $(V,B)$ with {\em simple system} $\Delta = \{\alpha_s \mid s \in S\}$ satisfying  the following conditions:
\begin{enumerate}
    \item  $\Delta$ is {\em positively linearly independent}: if $\sum_{\alpha \in \Delta} a_{\alpha} \alpha=0$ with $a_\alpha \geq 0$, then all $a_\alpha = 0$
\item if the order $m_{st}$ of the product $st$ is finite, then $B(\alpha_s,\alpha_t) = -\cos\left(\frac{\pi}{m_{st}}\right)$; while $B(\alpha_s,\alpha_t) \leq -1$ if and only $m_{st}$ is infinite.
\item for all $s\in S$, $B(\alpha_s,\alpha_s)=1.$ 
\end{enumerate}
Then $s\in S$ acts on $V$ as the reflection defined by $s(v)=v-2B(\alpha_s,v)\alpha_s$. Denote by $\Phi := W(\Delta)$ the associated {\em  root system}, which is the disjoint union of the sets of {\em positive roots } $\Phi^+=\cone(\Delta)\cap \Phi$ and {\em negative roots} $\Phi^-:=-\Phi^+$.
The pair $(\Phi,\Delta)$ is called a {\em based root system}, and its {\em rank} is $|\Delta| = |S|$. The {\em classical geometric representation} is obtained by assuming that $\Delta$ is a basis of $V$ and that $B(\alpha_s,\alpha_t) = -1$ whenever the order of $st$ in $W$ is infinite. For further details on the classical geometric representation, we refer the reader to [\cite[\S4.1]{BB}.

It is well-known that $\Phi^+$ is in bijection with $T$, see for instance~\cite[Proposition 4.4.5]{BB}:
$$
\Phi^+=\{\alpha_t\mid t \in T\}\quad\textrm{and}\quad T=\{s_\alpha \mid \alpha\in \Phi^+\},
$$
where $\alpha_t$ is the root associated to the reflection $s_{\alpha_t}=t$.

The \textit{inversion set of $w\in W$} is:
$$
\Phi(w)=\{\alpha\in \Phi^+ \mid w^{-1}(\alpha)\in\Phi^-\}.
$$
It is well-known that if $w=uv$ is a reduced product in $W$, then
$$
\Phi(w)=\Phi(u)\sqcup u(\Phi(v)),
$$
see for instance~\cite[Proposition~2.2]{DFHM24}. It follows in particular that:
\begin{itemize}
\item $u\leq_R w$ in the weak order if and only if $\Phi(u)\subseteq \Phi(w)$;
\item $\Phi(w)=\{\alpha_{s_1},s_1(\alpha_{s_2}),\dots, s_1\cdots s_{k-1}(\alpha_{s_k})\}$ if $w=s_1\cdots s_k$ is a reduced word for $w$. 
\end{itemize}

\subsection{The root poset}\label{ss:RootPoset} The \textit{depth function $\dep:\Phi^+\to \mathbb N$}  is defined by
$$
    \dep(\beta):=\min\{\ell(w)\mid w\in W, w(\beta)\in \Delta\}.
$$

The next result is \cite[Lemma~4.6.2]{BB} and describes how the depth of a root changes under the action of a simple reflection $s\in S$. 

\begin{Lemma}\label{lem:depth}
Let $s\in S$ and $\beta\in \Phi^+\setminus \{\alpha_s\}$. Then $\dep(\alpha_s)=0$ and

\begin{equation}\label{eq:depth}
\dep(s(\beta))=\begin{cases}
    \dep(\beta)-1, &\mbox{ if } B(\beta, \alpha_s)>0;\\
    \dep(\beta), &\mbox{ if } B(\beta, \alpha_s)=0;\\
    \dep(\beta)+1, &\mbox{ if } B(\beta, \alpha_s)<0.
\end{cases}
\end{equation}
\end{Lemma}

\begin{Remark} 
The definition for the depth of a positive root $\beta$ provided above is analogous to the one originally introduced in \cite[Definition 1.5]{BH} (see also \cite[Definition 4.6.1]{BB}), which is given by: 
$$
\dep'(\beta):=\min\{\ell(w)\mid w\in W, w(\beta)\in \Phi^-\}.
$$
This is equivalent to $\dep'(\beta)=\min\{\ell(w)\mid w\in W, w(\beta)\in -\Delta\}$. Indeed, if $w=s_1\cdots s_k$ is such that $w(\beta)\in \Phi^-$, then there exists an index $1\leq i\leq k$ such that $s_i\cdots s_k(\beta) \in \Phi^+$ while $s_{i-1}\cdots s_k(\beta) \in \Phi^-$.  Since $s_{i-1}$ permutes $\Phi^+\setminus \{\alpha_{s_{i-1}}\}$, it follows that $s_i\cdots s_k(\beta)=\alpha_{s_{i-1}}\in \Delta$. Hence, the two definitions are related by $\dep'(\beta)=\dep(\beta)+1$.
\end{Remark}

Recall from \cite[\S4.6]{BB} that $\Phi^+$ is equipped with a partial order $\leq$ called the \textit{root poset}.

\begin{Definition}
The \textit{root poset} $(\Phi^+, \leq)$ is defined as follows: $\alpha\leq \beta$ if and only if there exists $w\in W$ such that $w(\alpha)=\beta$ and $\ell(w)=\dep(\beta)-\dep(\alpha)$.     
\end{Definition}

Here we state well-known properties of the root poset, which can be found within~\cite[\S4.6~\&~\S4.7]{BB}.
\begin{Proposition}\label{prop:BasicDepth} The root poset $(\Phi^+,\leq)$ satisfies:
    \begin{itemize}
        \item[(i)] $(\Phi^+,\leq)$ is a graded poset with rank function $\dep: \Phi^+\rightarrow \mathbb{N}$.
        \item[(ii)] For $\alpha\in\Phi^+$, $\dep(\alpha)=0$ if and only if $\alpha\in\Delta$.
        \item[(iii)] The cover relations are given by $\alpha \lhd s(\alpha)$ with $s\in S$ and $\alpha \in \Phi^+\setminus \{\alpha_s\}$ such that $\dep(\alpha)<\dep(s(\alpha))$ (or equivalently, if $B(\alpha,\alpha_s)<0$). 
        \item[(iv)] We have $\alpha<\beta$ if and only if there exist $s_1,\ldots, s_k\in S$ such that $\alpha \lhd s_1(\alpha)\lhd s_2s_1(\alpha)\lhd \ldots \lhd s_k\cdots s_2s_1(\alpha)=\beta$. In this case, $\dep(\beta)-\dep(\alpha)=k=\ell(s_1\cdots s_k)$. 
    \end{itemize}
\end{Proposition}

\begin{Definition}\label{def:long}
    Let $s\in S$ and $\alpha \in \Phi^+\setminus\{\alpha_s\}$. A covering $\alpha \lhd s(\alpha)$ in the root poset $(\Phi^+, \leq)$ is called \textit{short} if $-1<B(\alpha, \alpha_s)<0$; otherwise, the covering is called \textit{long}. 
\end{Definition} 

By \cite[Proposition 4.5.4]{BB}, a dihedral subgroup generated by two reflections $s_{\alpha},s_{\beta}\in T$ is infinite if and only if $|B(\alpha,\beta)|\geq 1$, hence a covering $\alpha \lhd s(\alpha)$ is short (respectively, long) if and only if the dihedral group generated by $s_\alpha$ and $s$ is finite (respectively, infinite).

\begin{Example}\label{ex:root poset}
    Consider the Coxeter system as in Example \ref{ex:aut1}. The simple roots are $\Delta=\{\alpha_1, \alpha_2, \alpha_3\}$ and $\Phi^+=\cone(\Delta)\cap W(\Delta)$. In Figure \ref{fig:ex root poset gamma}, the Hasse diagram of the root poset for $\Phi^+$ is depicted up to depth equal to 3. The positive roots are expressed by their coordinates with respect to $\Delta$; for instance, $312$ means $3\alpha_1 + \alpha_2 + 2\alpha_3$. Moreover, each edge corresponds to a cover relation, and the colors depend on the simple reflection that is applied: orange corresponds to the simple reflection $s_1$, purple $s_2$ and green $s_3$. Finally, the edges are dashed when the cover is long, otherwise, the cover is short. For instance, $010\lhd 110=s_1(010)$ is a short covering since $B(\alpha_2, \alpha_1)=-1/2>-1$, while $B(\alpha_2+2\alpha_3, \alpha_1)=-1/2 -1 <-1$, hence $012\lhd 312=s_1(012)$ is a long cover relation, where $B$ is bilinear form of the canonical geometric representation.
\end{Example}
\begin{figure}[hbtp]
\centering
\includegraphics[scale=.9]{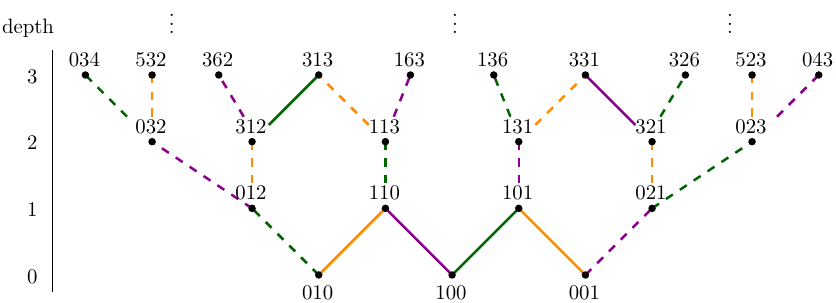}
\caption{Hasse diagram of the root poset of Example \ref{ex:root poset}.}
\label{fig:ex root poset gamma}
\end{figure}

\subsection{Small roots}\label{ss:SmallRoots} A \textit{small root\footnote{Also called an {\em elementary root} in the literature}} is a positive root that is reachable from a simple root by a chain of short covering relations in the root poset $(\Phi^+,\leq)$; see for instance \cite[\S 4.7]{BB}. We denote the set of small roots by $\Sigma$.  

\begin{Example} Consider $(W,S)$  of  Example \ref{ex:root poset}, then there are only five small roots:
$$
\Sigma=\{\alpha_1, \alpha_2, \alpha_3, \alpha_1+\alpha_2, \alpha_1+\alpha_3\}.
$$
\end{Example}

\begin{Remark}\label{rem:small-roots} In a finite Coxeter group, all positive roots are small. This follows from the fact that the root poset of a finite Coxeter system contains no long covering relation, since there are no infinite dihedral subgroups.
\end{Remark}

Small roots were introduced by Brink and Howlett in~\cite{BH} to show that Coxeter groups are automatic. The key point is the following statement; see also \cite[Theorem 4.7.3]{BB}.

\begin{Theorem}[{\cite[Theorem~2.8]{BH}}]\label{thm:BH} The set $\Sigma$ is finite  whenever $S$ is finite. 
\end{Theorem}

\subsection{The dominance order}\label{ss:dominance}

We present yet another partial order $\leq_d$ on the positive roots of $W$, called the \textit{dominance order}, which was also introduced by Brink and Howlett in \cite{BH}.   Let $\alpha,\beta\in\Phi^+$, we say that \textit{$\beta$ dominates $\alpha$}, denoted by $\alpha\leq_d \beta$, if and only if, for $w\in W$, $\beta\in \Phi(w)$ implies $\alpha\in \Phi(w)$. For $\beta\in \Phi^+$, the {\em dominance set of $\beta$} is:
   $$
   \dom(\beta)=\{\alpha\in \Phi^+\mid \alpha\leq_d \beta\}.
   $$
    The \textit{dominance-depth (or infinite-depth)} of a positive root $\beta\in\Phi^+$ is defined as follows:
$$
    \dep_\infty(\beta):=|\{\alpha\in \Phi^+\setminus\{\beta\}\mid\alpha\leq_d\beta\}|=|\dom(\beta)|-1.
$$ 
We have an analogous result of Lemma \ref{lem:depth}, see for instance \cite[Proposition 3.3]{DH}.

\begin{Proposition}\label{prop:inf-depth}
   Let $s\in S$ and $\beta\in\Phi^+\setminus\{\alpha_s\}$, then we have $\dep_\infty(\alpha_s)=0$ and
    $$\dep_\infty(s(\beta))=\begin{cases}
        \dep_\infty(\beta), &\mbox{ if } |B(\beta, \alpha_s)|<1;\\
    \dep_\infty(\beta)+1, &\mbox{ if } B(\beta, \alpha_s)\leq -1;\\
    \dep_\infty(\beta)-1, &\mbox{ if } B(\beta, \alpha_s)\geq 1.
    \end{cases}$$
\end{Proposition}

Note that if $\alpha\lhd s(\alpha)$ is a long covering in $(\Phi^+,\leq)$, then $\dep_\infty(s(\alpha))=\dep_\infty(\alpha)+1$, otherwise $\dep_\infty(s(\alpha))=\dep_\infty(\alpha)$.

The following characterization of the dominance order, which relates the dominance-depth to the depth of a positive root, is~\cite[Lemma 2.3]{BH}.

\begin{Proposition}\label{prop:dom}
    Let $\alpha,\beta\in\Phi^+$, then $\alpha\leq_d\beta$ if and only if $\dep(\alpha)\leq \dep(\beta)$ and $B(\alpha, \beta)\geq 1$; there is equality if and only if $\alpha=\beta$.
\end{Proposition}

It follows that if $\alpha\lhd s(\alpha)$ is a long covering in the root poset, then $\alpha_s\leq_d \alpha$. For instance, in Example \ref{ex:root poset}, $\alpha_3\leq_d \alpha_2+2\alpha_3 \leq_d \alpha_1+ 3\alpha_2 + 6 \alpha_3$. \\

\begin{Remark} The root poset and the dominance order are fundamentally different orders: one is graded by $\dep$ while the other is not graded by $\dep_\infty$. Also, for instance in Figure~\ref{fig:ex root poset gamma}, we have $100\lhd 110$ but $100\not <_d 110$ and $001<_d 012$ but $001\ntriangleleft 012$.
\end{Remark}

The next theorem also due to Brink and Howlett~\cite{BH} shows that a root is small if it does not dominate any other root but itself; see for instance~\cite[Proposition~4.7.6]{BB} for a proof.

\begin{Theorem}
\label{th:small-dp}
    Let $\alpha\in\Phi^+$, then $\alpha\in\Sigma$ if and only if $\dep_\infty(\alpha)=0$.
\end{Theorem}

\subsection{$m$-small roots and $m$-low elements}\label{ss:mSmallRoots}    We say that a positive root $\beta$ is {\em $m$-small} if $\dep_\infty(\beta)\leq m$, for some $m\in \mathbb{N}$. 

We denote by $\Sigma_m$ the set of $m$-small roots. Note that $\Sigma_0=\Sigma$ by Theorem~\ref{th:small-dp}. The following is a generalization of Brink-Howlett's Theorem~\ref{thm:BH} by Fu.

\begin{Theorem}[{\cite[Corollary~3.9]{Fu}}]
The set $\Sigma_m$ is finite for all $m\in \mathbb{N}$.  
\end{Theorem}

 For $w\in W$, define $\Sigma_m(w):=\Phi(w)\cap \Sigma_m$ the set of \textit{$m$-small inversion set} of $w$.

\begin{Definition} For $m\in \mathbb{N}$, an element $w \in W$ is \emph{$m$-low} if $\Phi(w) = \text{cone}(\Sigma_m(w))$. We denote by $L_m:=L_m(W)$ the set of $m$-low elements in $W$.
\end{Definition}

In Example~\ref{ex:aut1}, the Garside shadow $G$ is the set of low elements $L_0$.  The following theorem is due  to Dyer and the second author~\cite[Theorem 1.1]{DH} for $m=0$ and for affine Coxeter systems, and to Dyer \cite[Corollary 1.7]{Dyer} for all $m\in \mathbb N$ and any Coxeter system.

\begin{Theorem} For any Coxeter system and any $m\in \mathbb N$,  the set $L_m$ is a finite Garside shadow.
\end{Theorem}

We denote $\pi_m:W\mapsto  \bigvee_R \{g\in L_m\mid g\leq_R w\}$ the {\em Garside projection} on $L_m$. 
In~\cite{DFHM24}, the authors showed that the set of $m$-low elements are precisely the minimal elements (in right weak order) in $m$-Shi arrangements, that is, the subarrangement of the Coxeter arrangement constituted of the hyperplanes associated to $m$-small roots.  This implies in particular the following useful proposition. 

\begin{Proposition}\label{prop:mSmall} Let $w\in W$, then $\Sigma_m(\pi_m(w))=\Sigma_m(w)$.
\end{Proposition}
\begin{proof} Since $\pi_m(w)\leq_R w$, we have $\Sigma_m(\pi_m(w))\subseteq \Sigma_m(w)$. The converse follows from  \cite[Theorem 6.4]{DFHM24} in which the authors prove that the set 
$$
\{g\in W\mid \Sigma_m(g)=\Sigma_m(w)\}
$$
has a unique element of minimal length $x\in L_m$ and that $x\leq_R w$. So $\Sigma_m(w)=\Sigma_m(x)$. But $x\leq_R\pi_m(x)$ by definition. Therefore  $ \Sigma_m(w)=\Sigma_m(x)\subseteq \Sigma_m(\pi_m(w))$.
\end{proof}


\section{Reflection-prefixes and palindromic reduced words of reflections}\label{se:reflectionprefix}

In this section, we  first define and state properties of {\em reflection-prefixes} in relation to the root poset, to the dominance order and to the canonical generators of maximal dihedral subgroups.  Then we produce a family of automata built from the finite Garside shadows $L_m$ ($m\in\mathbb N)$  that recognize the language of reflection-prefixes. 

\subsection{Reflection-prefixes}\label{ss:reflectionprefixes}

 Let $t\in T$ be a reflection of $(W,S)$. It is well known that $\ell(t)=2k+1$ for some $k\in\mathbb N$ and that  if $t=s_1\dots s_k s_{k+1} s_{k+2} \dots s_{2k+1}$ is a reduced word for $t$, the word $s_1 \dots s_k s_{k+1} s_k\dots s_1$ is a {\em palindromic reduced word} for $t$; see for instance \cite[Proposition~2.3]{DFHM24}.  
 
 The following proposition shows in particular that $t$ is uniquely determined by the prefix $s_1\cdots s_ks_{k+1}$ of any of its palindromic reduced words. 
 
 \begin{Proposition}\label{prop:Pref1} Let $t\in T$ with corresponding positive root $\alpha_t$. Let  $w\in W$ and assume there is $r\in D_R(w)$ such that $t=wrw^{-1}$ and $\ell(t)=2\ell(w)-1$. Then $D_R(w)=\{r\}$. Moreover: 
 \begin{enumerate}

\item if $w=s_1\cdots s_k s_{k+1}$ is a reduced word for $w$, then $s_{k+1}=r$ and $t=s_1\cdots s_k r s_k\cdots s_1$ is a palindromic reduced word for $t$;
\item $\alpha_t=w(\alpha_r)= s_1\cdots s_k(\alpha_r)$ and  $\dep(\alpha_t)=k=\ell(w)-1$.
\end{enumerate}
\end{Proposition}
\begin{proof}   Let $w=s_1\cdots s_k r$ be a reduced word.  We show that $D_R(w)=\{r\}$. By contradiction, assume that $|D_R(w)|\geq 2$. So there exists  a $s\in D_R(w)$ with $s\not = r$; set $I=\{s,r\}\subseteq D_R(w)$.  By Lemma~\ref{lem:Descents},  we know that $W_I$ is finite and  there is $u\in X_I$ such that $w=u w_{\circ,I}$ is a reduced product. So $\ell(w)=\ell(u)+m$, where $m=m_{sr}\geq 2$. Now, observe that  $w_{\circ,I}rw_{\circ,I}\in I$, since the longest element of a finite Coxeter system acts by conjugation on the set of simple reflections. Therefore 
$$
t= s_1\cdots s_k s_{k+1} s_k\cdots s_1 =s_1\cdots s_k r s_k\cdots s_1=w r w^{-1}=uw_{\circ,I}rw_{\circ,I}u^{-1},
$$
implying that $\ell(t)\leq 2\ell(u)+1=2\ell(w)-2m+1<2\ell(w)-1=\ell(t)$, a contradiction. The remaining claims follow from the uniqueness of $r$, the definition of the depth of roots and from the equalities $\ell(t)=2\dep(\alpha_t)+1=2\ell(w)-1$.
\end{proof}

\begin{Definition}[Reflection-prefixes] A {\em reflection-prefix} is an element $p\in W$ such that for any reduced word $p= s_1 \cdots s_{k+1}$, the word $s_1 \cdots s_k s_{k+1} s_k\cdots s_1$ is a palindromic reduced word for some  reflection $t\in T$. In this case, we refer to $p$ as a {\em $t$-prefix}. 
\end{Definition}

Thanks to Proposition~\ref{prop:Pref1}, the reflection $t$ associated with a reflection-prefix $p$ does not depends on the choice of the reduced word for $p$. Since every $t\in T$ admits at least one palindromic reduced word, the set of $t$-prefixes is not empty. For a simple reflection $s\in S$, the unique $s$-prefix is $p_s=s$.

\begin{Example}[Dihedral groups] Let $W=I_2(m)$, with $m\in \mathbb{N}_{\geq 2}\cup \{\infty\}$ be the dihedral group generated by the simple reflections $\{r,s\}$ with $m_{rs}=m$. Recall that for two letters $a,b\in S$ and $k\in \mathbb N^*$, we denote $[ab]_k=abab\cdots$ the word with $k$-letters starting by $a$ and alternating the letters $a$ and $b$. We now distinguish two cases.
\begin{itemize}
    \item[(i)] Case $m$ is even. Any reflection $t\in T$ is written either as $t=[rs]_{2k+1}$ or $t=[sr]_{2k+1}$, for $0\leq k < \frac{m}{2}$. In the first case, the unique $t$-prefix is $p_t=[rs]_{k+1}$, while in the second case we have that the unique $t$-prefix is $p_t=[sr]_{k+1}$.
    \item[(ii)] Case $m$ is odd. Any reflection $t\in T$ is written as:
    \begin{equation*}
        t=\begin{cases}
            [rs]_{2k+1},& 0\leq k < \frac{m-1}{2};\\
            [sr]_{2k+1},& 0\leq k < \frac{m-1}{2};\\
            [rs]_m=[sr]_m = w_{\circ, \{r,s\}}.
        \end{cases}
    \end{equation*}
    In the first two cases, there is a unique $t$-prefix as in (i). In the last case, there are two $t$-prefixes: $p_t=[rs]_{\frac{m-1}{2}}$ and $p_t'=[sr]_{\frac{m-1}{2}}$.
\end{itemize}
\end{Example}

\begin{Example}[Universal Coxeter system]
    Let $(U_n, S)$ be the Universal Coxeter system of rank $n$, that is $U_n=\langle S\mid s^2=e \ \mbox{for all $s \in S$} \rangle$. Every element has a unique reduced word, so the language of palindromic reduced words for $(U_n,S)$ is given by $\RPal= \{s_1\cdots s_{i-1}s_is_{i-1}\cdots s_1 \mid s_j\in S \mbox{ and } s_j\neq s_{j+1}\}$. Furthermore, the language of reduced words for reflection-prefixes coincides with the set of all reduced words in $W$, that is, $\Pref_T=\Red$.
\end{Example}

\begin{Example}\label{ex:a3}
In a Coxeter system $(W,S)$ of type $A_3$, consider the reflection\footnote{For simplicity of notations, we write $i$ in place of $s_i$.} $t=12321$. 
\begin{center} 
\begin{tikzpicture}
	[scale=2,
	 q/.style={teal,line join=round},
	 racine/.style={blue},
	 racinesimple/.style={blue},
	 racinedih/.style={blue},
	 sommet/.style={inner sep=2pt,circle,draw=black,fill=blue!40,thick,anchor=base},
	 rotate=0]
 \tikzstyle{every node}=[font=\small]
\def\grosseursimple{0.025}
\coordinate (ancre) at (0,3);

\node[sommet,label=below:$2$] (a2) at ($(ancre)+(0.4,0)$) {};
\node[sommet,label=below :$3$] (a3) at ($(ancre)+(0.8,0)$) {} edge[thick] node[auto,swap,right] {}(a2) ;
\node[sommet,label=below:$1$] (a4) at (ancre) {}  edge[thick] node[auto,swap,left] {} (a2);
\end{tikzpicture}
\end{center}

\noindent
This reflection admits other reduced expressions, namely  $t=12321=13231=31213=32123$. Correspondingly, it has has three distinct $t$-prefixes: $p_t=123$, $p_t'=132=312$, and $p_t''=321$. 
\end{Example}
 
\begin{Remark}
    If $t\in T$ and $p_t$ is a $t$-prefix, it is not true in general that there is an $r$-prefix $p_r$ with $p_r\leq_R p_t$ for all $\alpha_r\in \Phi(p_t)$. For instance, consider a Coxeter system of type $H_3$, with associated Coxeter graph:
\begin{center}
\begin{tikzpicture}
	[scale=2,
	 q/.style={teal,line join=round},
	 racine/.style={blue},
	 racinesimple/.style={blue},
	 racinedih/.style={blue},
	 sommet/.style={inner sep=2pt,circle,draw=black,fill=blue!40,thick,anchor=base},
	 rotate=0]
 \tikzstyle{every node}=[font=\small]
\def\grosseursimple{0.025}
\coordinate (ancre) at (0,3);

\node[sommet,label=below:$2$] (a2) at ($(ancre)+(0.4,0)$) {};
\node[sommet,label=below :$3$] (a3) at ($(ancre)+(0.8,0)$) {} edge[thick] node[auto,swap,right] {}(a2) ;
\node[sommet,label=below:$1$] (a4) at (ancre) {}  edge[thick] node[auto,swap,left, label=above:$5$] {} (a2);
\end{tikzpicture}
\end{center}

 Take $t=121232121 \in T$, $p_t=12123$ and $r=212$. Then $\alpha_r=2(\alpha_1)\in \Phi(p_t)$, in fact $p_t^{-1}(2(\alpha_1))=32121(2(\alpha_1))=312121(\alpha_1)\in \Phi^-$, but $212\not \leq_R 12123$.
\end{Remark}

\begin{Remark} $\ $
\begin{enumerate}
\item  Proposition~\ref{prop:Pref1} shows that any reflection-prefix is a {\em Grassmannian element}, that is, an element with at most one right descent. It would be interesting to characterize the subset of Grassmannian elements that are reflection-prefixes.  
\item Any $t$-prefix $p_t$  satisfies by definition $p_t\leq_R t$ in right weak order $(W,\leq_R)$.
\item Since $t^{-1}=t$, the notion of $t$-suffix exists and all the results valid for $t$-prefixes are valid for $t$-suffixes.
\end{enumerate}
\end{Remark}

The following statement follows from Proposition~\ref{prop:Pref1} and of the fact that $\ell(t)=2\dep(\alpha_t)+1$ for any $t\in T$.

\begin{Corollary}\label{cor:Pref1} Let $w\in W$ and  $w=s_1\cdots s_k r$ be a reduced word. Let $t\in T$, then the following statements are equivalent:
\begin{enumerate}
\item $w$ is a $t$-prefix;
\item   $t=s_1\cdots s_k r s_k\cdots s_1$ is a palindromic reduced word.
\item $\alpha_t=s_1\cdots s_k(\alpha_r)$ and $\dep(\alpha_t)=\ell(w)-1$.
\end{enumerate}
In particular, $\alpha_t\in \Phi(p_t)$ for any $t$-prefix $p_t$.
\end{Corollary}

The above corollary implies this useful consequence. 

\begin{Proposition}\label{prop:PrefDepth} Let $t\in T$ with associated positive root $\alpha_t\in \Phi^+$. For any $t$-prefix $p_t$ and any $\alpha\in \Phi(p_t)$, we have $\dep(\alpha)\leq \dep(\alpha_t)$; with equality if and only if $\alpha=\alpha_t$. 
\end{Proposition}
\begin{proof} Consider a reduced word $p_t=s_1\cdots s_kr$ for a $t$-prefix $p_t$. 
By Corollary~\ref{cor:Pref1}, we have $\dep(\alpha_t)=k$.
Now, let $\alpha\in \Phi(p_t)$, such that $\alpha \neq \alpha_t$.
Then there is $1\leq i\leq k-1$ such that $\alpha=s_1\cdots s_i(\alpha_{s_{i+1}})$. By definition of the depth we have
$\dep(\alpha)\leq i-1$. Therefore $\dep(\alpha)<\dep(\alpha_t)$, which proves the inequality for all $\alpha\neq \alpha_t$.
\end{proof}

\subsection{Reflection-prefixes and the root poset}

The following statement shows that there is a correspondence between reduced words of reflection-prefixes and saturated chains in the root poset. 

\begin{Proposition}\label{prop:PrefRoot} Let $t\in T$ and $\alpha_t\in \Phi^+$ be the corresponding positive root. 
\begin{enumerate}
\item Let  $\alpha_{r}\lhd s_k(\alpha_{r})\lhd\dots\lhd s_1\cdots s_{k-1} s_k(\alpha_r)=\alpha_t$ be a saturated chain in $(\Phi^+,\leq)$, where $r\in S$. Then $p_t=s_1\dots s_kr$ is a reduced word for some $t$-prefix $p_t$. 

\item Let $p_t$ be a  $t$-prefix and fix a reduced word $p_t=s_1\cdots s_k r $, with $r$ the unique right-descent. Then $\alpha_{r}\lhd s_k(\alpha_{r})\lhd\cdots\lhd s_1\cdots s_{k-1} s_k(\alpha_r)=\alpha_t$ is a saturated chain in $(\Phi^+,\leq)$.
\end{enumerate}
\end{Proposition}
\begin{proof}  (1) Let $r\in S$ and $\alpha_{r}\lhd s_k(\alpha_{r})\lhd\dots\lhd s_1\cdots s_{k-1} s_k(\alpha_r)=\alpha_t$ be a saturated chain in $(\Phi^+,\leq)$. So $\dep(\alpha_t) = k$. Let $p_t=s_1\cdots s_kr$ and observe that $\alpha_t=s_1\cdots s_k(\alpha_r)$. If $s_1\cdots s_k$ is not reduced, then $\dep(\alpha_t)<k$, which is a contradiction. So $s_1\cdots s_k$ is reduced. If $s_1\cdots s_kr$ is not reduced, then $\ell(s_1\cdots s_k r)<\ell(s_1\cdots s_k)$. Therefore 
$
\alpha_t=s_1\cdots s_k(\alpha_r)\in \Phi^-,
$
again a contradiction. So $s_1\cdots s_kr$ is a reduced word for $p_t$. Since $\ell(t)=2\dep(\alpha_t)+1$, the palindromic word  $t= s_{s_1\cdots s_k(\alpha_r)}=s_1\cdots s_krs_k\cdots s_1$ is reduced. Hence $p_t$ is a $t$-prefix. 

\smallskip
\noindent (2)  Let $p_t$ be a  $t$-prefix and choose a reduced word $p_t=s_1\cdots s_k r $. Recall that by Proposition~\ref{prop:Pref1}, $r$ is the unique right descent of $p_t$. By Corollary~\ref{cor:Pref1}, the palindromic word
$$
s_1\cdots s_{i-1}s_i \cdots s_kr s_k \cdots s_i s_{i-1} \cdots s_1
$$
is reduced. So for any $1\leq i\leq k$, the palindromic word $s_i \cdots s_kr s_k \cdots s_i$ is also reduced.

Assume now by contradiction that there is $1\leq i\leq k-1$ such that 
$s_{i}\cdots s_k(\alpha_{r})$ is not a cover of $s_{i+1}\cdots s_k(\alpha_{r})$. In other words, by Proposition~\ref{prop:BasicDepth}, $\dep(s_{i}\cdots s_k(\alpha_{r}))\leq \dep(s_{i+1}\cdots s_k(\alpha_{r}))$. Take $i$ maximal for that property, that is,  $\alpha_{r}\lhd s_k(\alpha_{r})\lhd\dots\lhd s_{i+1}\cdots  s_k(\alpha_r)$. So $\dep(s_{i+1}\cdots s_k(\alpha_{r}))=k-i$. Therefore $\dep(s_{i}\cdots s_k(\alpha_{r}))\leq k-i<\ell(s_i\cdots s_k)=k-i+1$. By Corollary~\ref{cor:Pref1}, the element $s_1\cdots s_kr$ is not a reflection-prefix and the palindromic word $s_{i}\cdots s_k r s_k\cdots s_i$ is not a reduced word, which is a contradiction.
 \end{proof}

 \begin{Example}\label{ex:a3-chains} Consider the Coxeter system $(W,S)$ of type $A_3$ and take $t=12321$. In Example \ref{ex:a3}, we have seen that the reflection $t$ admits 3 $t$-prefixes: $p_t=123$, $p_t'=132=312$, and $p_t''=321$. By Proposition \ref{prop:PrefRoot}, one can individuate these $t$-prefixes by looking at the saturated chains in the root poset of $W$, from a simple root to $\alpha_t=\alpha_1+\alpha_2+\alpha_3$. In Figure \ref{fig:chains}, the saturated chains associated to the $t$-prefixes are highlighted. Note that $p_t'$ has two reduced expressions, hence there are two corresponding saturated chains starting from $\alpha_2$.
\begin{figure}[h]
    \centering    \includegraphics[width=1\linewidth]{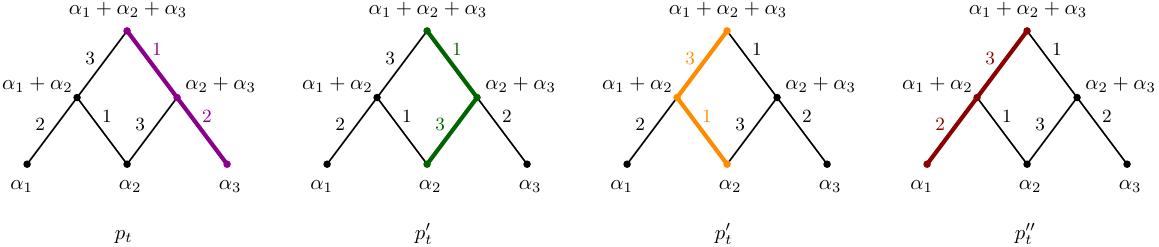}
    \caption{The saturated chains corresponding to the $t$-prefixes of Example \ref{ex:a3-chains}.}
    \label{fig:chains}
\end{figure}

\end{Example}
 
 Recall that the left weak order $(W,\leq_L)$ is defined by $u\leq_L v$ if and only if $u$ is a suffix of $v$. 

\begin{Corollary}\label{cor:Pref2} Let $t\in T$ with associated root $\alpha_t\in \Phi^+$. Consider $p_t$ be a $t$-prefix with $r$ its unique right descent. Let $u\in W$ such that  $e<_Lu\leq_L p_t$ in $(W,\leq_L)$, then $-u(\alpha_r)\leq \alpha_t$ in  $(\Phi^+,\leq)$.
\end{Corollary}
\begin{proof}  Since $e<_L u\leq_L p_t$, there is a reduced word $p_t=s_1\cdots  s_k r$ and $1\leq i\leq k$ such that $u=s_i\cdots s_k r$. So $-u(\alpha_r)=s_i\cdots s_k(\alpha_r)\leq s_1\cdots s_k(\alpha_r)=-p_t(\alpha_r)=\alpha_t$ by Proposition~\ref{prop:PrefRoot}.
\end{proof}

The following theorem provides a characterization of reflection-prefixes in term of inversion sets.

\begin{Theorem}\label{thm:Pref1}  Let $w\in W$ and $r\in D_R(w)$. Let $t\in T$ such that $\alpha_t=-w(\alpha_r)\in \Phi(w)$. Then $w$ is a $t$-prefix if and only if $B(\alpha,\alpha_t)>0$ for all $\alpha\in \Phi(w)$. 
\end{Theorem}
\begin{proof}  Assume first that $w=p_t$ is a $t$-prefix. By Proposition~\ref{prop:PrefRoot}, we have
$$
\alpha_{r}\lhd s_k(\alpha_{r})\lhd\dots\lhd s_1\cdots s_{k-1} s_k(\alpha_r)=\alpha_t
$$ 
 a saturated chain associated to the reduced word $p_t=s_1 \cdots s_kr$.  By the characterization of covers in the root poset, see Proposition~\ref{prop:BasicDepth}, we have: 
$$
 B(s_{i+1}\cdots s_k(\alpha_{r}),\alpha_{s_{i}})<0,\quad \textrm{for all} \ 1\leq i\leq k-1.
$$
Since $p_t=s_1\cdots s_kr$ is a reduced word, we know that the inversion set of $p_t$ is:
$$
\Phi(w)=\{\alpha_{s_1},s_1(\alpha_{s_2}),\dots,s_1\cdots s_{k-1}(\alpha_{s_k}),s_1\cdots s_k(\alpha_r)=\alpha_t\}
$$
Let $\alpha\in \Phi(p_t)$. If $\alpha=\alpha_t$ then $B(\alpha_t,\alpha_t)>0$. If $\alpha=\alpha_{s_1}$, then
$$
B(\alpha_t,\alpha_{s_1})=B(s_1\cdots s_k(\alpha_{r}),\alpha_{s_{1}})= -B(s_2\cdots s_k(\alpha_{r}),\alpha_{s_{1}})>0.
$$
If $\alpha\notin\{ \alpha_t,\alpha_{s_1}\}$, then there is $2\leq i\leq k-1$ such that $\alpha=s_1\cdots s_{i-1}(\alpha_{s_{i}})$, then 
$$
 B(\alpha_t,\alpha)=B(s_1\cdots s_k(\alpha_{r}),s_1\cdots s_{i-1}(\alpha_{s_{i}}))=- B(s_{i+1}\cdots s_k(\alpha_r),\alpha_{s_{i}})>0.
$$
Therefore   $B(\alpha,\alpha_t)>0$ for all $\alpha\in \Phi(p_t)$.

\smallskip

Conversely, assume that $B(\alpha,\alpha_t)>0$ for all $\alpha\in \Phi(w)$. Consider a reduced word $w=s_1\cdots s_k r$, so that $\alpha_t=s_1\cdots s_k(\alpha_r)$. To show that $w$ is a $t$-prefix, it suffices to prove that:
$$
 B(s_{i+1}\cdots s_k(\alpha_{r}),\alpha_{s_{i}})<0,\quad \textrm{for all} \ 1\leq i\leq k-1.
$$
This follows directly from the same computations done in the first part of the proof. 
\end{proof}

\subsection{Reflection-prefixes and the dominance order}

Reflection-prefixes allow an efficient way to compute the dominance set of  $\beta\in \Phi^+$: 
$$
\dom(\beta)=\{\alpha \in \Phi^+\mid \alpha\leq_d \beta\}
$$
and its infinite-depth $\dep_\infty(\beta)$; see \S\ref{ss:dominance}.

\begin{Proposition}\label{prop:PrefDom} Let $\beta \in \Phi^+$ and let $t=s_\beta$ be the corresponding reflection. For any $t$-prefix $p_t$ we have:  
\begin{eqnarray*}
\dom(\beta) & = & \{ \alpha\in \Phi(p_t) \ \mid\  B(\alpha,\beta)\geq 1\}\\
&=&\{\beta\}\sqcup \{ \alpha_s\in \Phi(p_t) \ \mid\  |\langle s,t\rangle| =\infty\}.
\end{eqnarray*}
In particular, $\dep_\infty (\beta)=|\{ \alpha_s\in \Phi(p_t) \ \mid\ |\langle s,t\rangle| =\infty\}|$.
\end{Proposition}
\begin{proof} The second equality follows from the well-known fact that a dihedral subgroup containing $s,t$ is infinite if and only if $|B(\alpha_s,\alpha_t)|\geq 1$ and $s\not = t$,  see for instance \cite[Proposition 4.5.4]{BB}.   

We show the first equality. Set $A=\{ \alpha\in \Phi(p_t) \ \mid\  B(\alpha,\beta)\geq 1\}$. Let $\alpha\in A$, then $\dep(\alpha)\leq \dep(\alpha_t)=\dep(\beta)$ by Proposition~\ref{prop:PrefDepth}. By Proposition~\ref{prop:dom}, we have therefore $\alpha\in \dom(\beta)$. Let $\alpha\in \dom (\beta)$, then $ B(\alpha,\beta)\geq 1$ by Proposition~\ref{prop:dom} again. Now,  $\beta\in \Phi(p_t)$ by Corollary~\ref{cor:Pref1}. So $p_t^{-1}(\beta)\in \Phi^-$. Since $\alpha\leq_d\beta$, we have by definition of the dominance order that $p_t^{-1}(\alpha)\in \Phi^-$. In other words $\alpha\in \Phi(p_t)$, which completes the proof.
\end{proof}

\begin{Example} Consider the Coxeter system of Example \ref{ex:aut1}. Let $t=2321232\in~T$ and $p_t=2321$, and call $\beta=\alpha_t$. 
We have $B(\alpha_1,\alpha_2)=B(\alpha_1,\alpha_3)=-1/2$ and $B(\alpha_2,\alpha_3)=-1$.  Hence, $\Phi(p_t)=\{\alpha_2, 2\alpha_2+\alpha_3, 3\alpha_2+2\alpha_3, \beta=\alpha_1+6\alpha_2+3\alpha_3\}$. Finally, to obtain the set of dominance of $\beta$, we have to compute the bilinear form $B$ between $\beta$ and the roots of $\Phi(p_t)\setminus\{\beta\}$ and see when the value is equal or grater than 1. Therefore it is not difficult to check that $\dom(\beta)=\{\alpha_2, 2\alpha_2+\alpha_3\}$, and $\dep_\infty(\beta)=2$.  
\end{Example}

\subsection{Reflection-prefixes and dihedral reflection subgroups} In this section, we show how to use reflection-prefixes to compute the canonical generators of any dihedral reflection subgroup $W'$ of $W$. This approach leads to an algorithm in \S\ref{ss:Algo} that does not need a geometric representation to operate. 
\smallskip

Let $W'$ be a reflection subgroup of $W$, i.e., $W'=\langle A\rangle$ for some $A\subseteq T$. The set of reflections in $W'$ is $T_{W'}=T\cap W'$. Write:
$$
\Phi_{W'}=\{\alpha \in \Phi \mid s_\alpha \in W'\}\quad \textrm{and}\quad \Delta_{W'}=\{\alpha\in \Phi^+ \mid \Phi(s_\alpha)\cap \Phi_{W'}=\{\alpha\}\}.
$$
Dyer~\cite{Dy90} showed that $(W',\chi(W'))$ is a Coxeter system with set of {\em canonical generators}:
$$
\chi(W')=\{r\in T_{W'}\mid \ell(tr)<\ell(r),\ \forall t\in T_{W'}\}=\{s_\alpha \in T\mid \alpha\in \Delta_{W'}\}.
$$
Moreover, he showed that $\Phi_{W'}$ is a root system for $(W',\chi(W'))$, with simple system  $\Delta_{W'}$ and  set of positive roots is $\Phi_{W'}^+=\Phi^+\cap \Phi_{W'}$. In particular $T_{W'}=\{s_\alpha\mid \alpha\in\Phi_{W'}^+\}$. We denote by $\ell_{W'}:W'\to \mathbb N$ the  length function relative to $(W',\chi(W'))$. We refer the reader to \cite[\S2.8~\&~\S2.9]{DFHM24} for a survey of the properties used in this section.

Any element $w\in W$ has a unique decomposition $w=uv$ with $u\in W'$, $v\in X_{W'}$, where $X_{W'}=\{x\in W\mid \ell(s_\alpha x) >\ell(x),\ \forall \alpha\in\Delta_{W'}\}$ is the set of {\em minimal coset representatives of $W'\backslash W$}. Moreover, $\Phi_{W'}(u)$, the inversion set of $u$ relative to the Coxeter system $(W',\chi(W'))$, satisfies the following equation: 
$$
\Phi_{W'}(u)=\Phi(w)\cap \Phi_{W'}=\Phi(u)\cap\Phi_{W'}\ \text{and}\ |\Phi_{W'}(u)|=\ell_{W'}(u).
$$

 \smallskip
 
 From now on, fix  $r,t\in T$ such that $r\not =t$ and let  $W'=\langle r,t\rangle$ be the {\em dihedral reflection subgroup} generated by $r$ and $t$, with set of canonical generators $\chi(W')=\{s_1,s_2\}$ and simple system $\Delta_{W'}=\{\alpha_1,\alpha_2\}$. Let $m=\ord(s_1 s_2)\in \mathbb N_{\geq 2}\sqcup\{\infty\}$. Recall that for two letters $s,t\in T$ and $k\in \mathbb N^*$, we denote $[st]_k=stst\cdots$ the word with $k$-letters starting by $s$ and alternating the letters $s$ and $t$. For $0\leq k< m/2$, write $t_{1,k}=[s_1s_2]_{2k+1}$ and $t_{2,k}=[s_2s_1]_{2k+1}$. The set of reflections $T_{W'}$ admits a partition $T_{W'}=T_{W'}^1 \sqcup T_{W'}^2$, where:
$$
T_{W'}^1=\{t_{1,k}\mid 0\leq k < m/2\} \quad\textrm{and}\quad T_{W'}^2=\{t_{2,k}\mid 0\leq k < m/2\}.
$$
For  all $0\leq k < m/2$, observe that the words $t_{1,k}=[s_1s_2]_{2k+1}$ and $t_{2,k}=[s_2s_1]_{2k+1}$ are reduced and that $\ell_{W'}(t_{i,k})=2k+1$. We consider the following total order on $T_{W'}$:
$$
s_1=t_{1,0}\prec \dots \prec t_{1,k}\prec t_{1,k+1} \prec\dots \prec  t_{2,l+1} \prec t_{2,l}\prec \dots t_{2,1}\prec t_{2,0}=s_2,
$$
where $0\leq k,l < m/2-1$. This is a {\em reflection order} on $T_{W'}$ in the sense of~\cite[Definition~2.1]{Dy92}; see also \cite[\S5.2]{BB}.

Write $\alpha_{i,k}\in \Phi^+_{W'}$ for the root corresponding to the reflection $t_{i,k}$. We obtain therefore a partition $\Phi^+_{W'}=\Phi_{W'}^1 \cup \Phi_{W'}^2$, where $\Phi^i_{W'}=\{\alpha_{i,k}\mid 0\leq k<m/2\}$; moreover $\Phi_{W'}^1\cap\Phi_{W'}^2$ is $\{ \alpha_{1,(m-1)/2}=\alpha_{2,(m-1)/2}\}$ if $m$ is odd or empty  if $m$ is infinite or even. This decomposition yields a total order on $\Phi^+_{W'}$:
\begin{equation}
\label{eq:InitialSection}
\alpha_1=\alpha_{1,0}\prec \dots \prec \alpha_{1,k}\prec \alpha_{1,k+1} \prec\dots \prec  \alpha_{2,l+1} \prec \alpha_{2,l}\prec \dots \alpha_{2,1}\prec \alpha_{2,0}=\alpha_2,
\end{equation}
where $0\leq k,l < m/2-1$. The following proposition shows that restriction of reflection-prefixes to dihedral subgroups are reflection-prefixes.

\begin{Proposition}\label{prop:PrefDihed} Let $t\in T$ and $W'$ be a dihedral reflection subgroup containing $t$, with set of canonical generators $\chi(W')=\{s_1,s_2\}$ and $m=\ord(s_1 s_2)$. Let $p_t$ be a $t$-prefix. Write $p_t=u_t v$ with $u_t\in W'$ and $v\in X_{W'}$. Then $u_t$ is a $t$-prefix for the Coxeter system $(W',\chi(W'))$ and either $\Phi_{W'}(u_t)\subseteq \Phi_{W'}^1$ or $\Phi_{W'}(u_t)\subseteq \Phi_{W'}^2$.
 \end{Proposition}

We need two lemmas before proving the above proposition. First, we need the following useful property that is a consequence of Dyer's ``functoriality of the Bruhat graph''~\cite{Dy91}: for any $s,t\in T_{W'}$ such that $\ell_{W'}(s)< \ell_{W'}(t)$ we have $\ell(s)< \ell(t)$, or in other words, $\dep(\alpha_s)<\dep(\alpha_t)$; see for instance~\cite[\S2.8-Eq.(4)]{DFHM24}. From this discussion, we deduce the following lemma. 

\begin{Lemma}\label{lem:DepthDi} With $\Phi_{W'}^+$ totally ordered as above, the function $\dep:\Phi_{W'}\to \mathbb N$ is strictly increasing on $\Phi_{W'}^1$ and strictly decreasing on $\Phi_{W'}^2$. \qed
\end{Lemma}

Second, Dyer showed in \cite{Dy92} that inversion sets of elements in $W$ are {\em initial sections} of reflection orders. In our case,  it implies that any inversion set $\Phi_{W'}(u)$ is either a finite segment starting at $\alpha_1$ or starting at $\alpha_2$, that is in Eq~(\ref{eq:InitialSection}), a saturated chain starting at $\alpha_1$ or ending at $\alpha_2$. Moreover $\Phi_{W'}(u)$ contains  both $\alpha_1$ and $\alpha_2$  if and only if $u$ is the longest element $w_{\circ,W'}$ of $W'$ and $W'$ is therefore finite. We give this statement in the following proposition, which follows directly from~\cite[Definition~2.1~\&~Lemma~2.11]{Dy92} by considering the above reflection order on $T_{W'}$ or its dual.

\begin{Lemma}\label{lem:InvDi} Let $u\in W'$, then there is $0\leq k< m/2$ such that one of the following equality is satisfied:  
\begin{enumerate} 
\item $\Phi_{W'}(u)=\{\alpha_1=\alpha_{1,0},\alpha_{1,1}, \dots , \alpha_{1,k} \}$;
\item $\Phi_{W'}(u)=\{\alpha_2=\alpha_{2,0},\alpha_{2,1}, \dots , \alpha_{2,k} \}$;
\item $m<\infty$ is odd, $k<m/2-1$ and 
$$
\Phi_{W'}(u)=\{\alpha_1=\alpha_{1,0}, \dots , \alpha_{1,(m-1)/2}=\alpha_{2,(m-1)/2},\alpha_{2,(m-3)/2},\dots ,\alpha_{2,k} \} 
$$
$$
\textrm{or}\quad \Phi_{W'}(u)=\{\alpha_2=\alpha_{2,0}, \dots , \alpha_{2,(m-1)/2}=\alpha_{1,(m-1)/2},\alpha_{1,(m-3)/2},\dots ,\alpha_{1,k} \} 
$$

\item $m<\infty$ is even and 
$$
\Phi_{W'}(u)=\{\alpha_1=\alpha_{1,0}, \dots , \alpha_{1,m/2-1},\alpha_{2,m/2-1},\alpha_{2,(m-2)/2},\dots ,\alpha_{2,k} \} 
$$
$$
\textrm{or}\quad \Phi_{W'}(u)=\{\alpha_2=\alpha_{2,0}, \dots , \alpha_{2,m/2-1},\alpha_{1,m/2-1},\alpha_{1,(m-2)/2},\dots ,\alpha_{1,k} \} 
$$
\end{enumerate}
Moreover, $\alpha_1,\alpha_2\in \Phi_{w'}(u)$ if and only if $u=w_{\circ,W'}$. \qed
\end{Lemma}

\begin{proof}[Proof of Proposition~\ref{prop:PrefDihed}] Let $r\in S$ such that $D_R(p_t)=\{r\}$, by Proposition~\ref{prop:Pref1}. Write $p_t=wr$ as a reduced product. If $w=e$, then $p_t=r$ and $t=r\in\chi(W')$. So $t$ is the $t$-prefix in $(W',\chi(W'))$. Assume  $w\not = e$. Then $\alpha_t=-p_t(\alpha_r)=w(\alpha_r)$ and $\Phi(w)=\Phi(p_t)\setminus\{\alpha_t\}$. So $\alpha_t$ is the leftmost root in each expression for $\Phi_{W'}(u_t$ in Lemma~\ref{lem:InvDi}). But $\dep(\alpha_t)$ is the unique maximum of the depth of roots in $\Phi(p_t)$, by Proposition~\ref{prop:PrefDepth}. Hence  $\dep(\alpha_t)$ is the unique maximum of the depth on $\Phi_{W'}(u_t)$. So by Lemma~\ref{lem:DepthDi} the only possibilities for $\Phi_{W'}(u_t)$ are:

\noindent (a) $\Phi_{W'}(u_t)=\{\alpha_1=\alpha_{1,0},\alpha_{1,1}, \dots , \alpha_{1,k} \}$, for some $1\leq k< m/2$, and $\alpha_t=\alpha_{1,k}$. Then $\Phi_{W'}(u_t)\subseteq \Phi_{W'}^1$. So  $u_{t}=[s_1 s_2]_{k}$ is a $t$-prefix in $(W',\chi(W'))$ by definiton.

\noindent (b) $\Phi_{W'}(u)=\{\alpha_2=\alpha_{2,0},\alpha_{2,1}, \dots , \alpha_{2,k} \}$, for some $1\leq k< m/2$,  and $\alpha_t=\alpha_{2,k}$. Then  $\Phi_{W'}(u_t)\subseteq \Phi_{W'}^2$ and $u_{t}=[s_2 s_1]_{k}$ is a $t$-prefix in $(W',\chi(W'))$. 
\end{proof}

\subsection{Reflection-prefixes and canonical generators of dihedral reflection subgroups}\label{ss:Algo}

As a consequence of Proposition~\ref{prop:PrefDihed} and the description of inversion sets in Lemma~\ref{lem:InvDi} we obtain the following   algorithm to compute the canonical generators of a dihedral reflection subgroup and its simple system. This algorithm is for instance useful when one wants to find the canonical set of generators of a reflection subgroups as explained in~\cite[Proposition~3.7]{Dy90}.

\begin{Theorem}\label{prop:algo} Let $r,t\in T$ with $r\not = t$ and let $W'=\langle r,t\rangle$ be the dihedral reflection subgroup generated by $r$ and $t$. Let $p_r$ be a $r$-prefix and $p_t$ be a $t$-prefix.
\begin{enumerate}
\item If $|\Phi_{W'}(p_r)|=|\Phi_{W'}(p_t)|=1$, then:
$$
\Delta_{W'}=\{\alpha_r,\alpha_t\}\quad\textrm{and}\quad\chi(W')=\{r,t\}.
$$
\item If   $|\Phi_{W'}(p_r)|>1$ (resp. $|\Phi_{W'}(p_t)|>1$), then:
 $$
 \Delta_{W'}=\{\alpha_1,\alpha_2\}\quad\textrm{and}\quad\chi(W')=\{s_1,s_2\},
 $$
 where $\alpha_1$ is the unique root with the smallest depth in $\Phi(p_r)\cap \Phi_{W'}$ (resp. $\Phi(p_t)\cap \Phi_{W'}$) and $s_1(\alpha_2)$ is the unique root with smallest depth in $\Phi(p_r)\cap \Phi_{W'}\setminus\{\alpha_1\}$ (resp. $\Phi(p_t)\cap \Phi_{W'}\setminus\{\alpha_1\}$). \qed
\end{enumerate}
 \end{Theorem}

The advantage of the algorithm described in the theorem above is that one has only to compute the roots $\Phi(p_r)$ and $\Phi(p_t)$, and not to compute first a subset of positive roots up to a certain depth.

 \begin{Example} Let $(W,S)$ be of type 
\begin{center}
\begin{tikzpicture}
	[scale=2,
	 q/.style={teal,line join=round},
	 racine/.style={blue},
	 racinesimple/.style={blue},
	 racinedih/.style={blue},
	 sommet/.style={inner sep=2pt,circle,draw=black,fill=blue!40,thick,anchor=base},
	 rotate=0]
 \tikzstyle{every node}=[font=\small]
\def\grosseursimple{0.025}
\coordinate (ancre) at (0,3);

\node[sommet,label=above:$1$] (a2) at ($(ancre)+(0.25,0.4)$) {};
\node[sommet,label=below right :$3$] (a3) at ($(ancre)+(0.5,0)$) {} edge[thick] node[auto,swap,right] {}(a2) ;
\node[sommet,label=below left:$2$] (a4) at (ancre) {} edge[thick] node[auto,swap,below] {$4$} (a3) edge[thick] node[auto,swap,left] {} (a2);
\end{tikzpicture}
\end{center}

Take $t=3132313$ and $r=3123213$, and let $W'=\langle t,r\rangle$. Consider the $t$-prefix $p_t=3132$ and the $r$-prefix $p_r=3123$. In order to apply the algorithm of Theorem \ref{prop:algo}, we need to compute the sets $\Phi_{W'}(p_r)$ and $\Phi_{W'}(p_t)$. First of all, the inversion sets of $p_t$ and $p_r$ are, respectively, $\Phi(p_t)=\{\alpha_3, \alpha_1+\alpha_3, \alpha_1, \alpha_t\}$ and $\Phi(p_r)=\{\alpha_3, \alpha_1+\alpha_3, 31(\alpha_2),\alpha_t\}$. (By Lemma \ref{lem:DepthDi}, we just need to compute the roots in $\Phi^+_{W'}$ of depth at most $\max(\dep(\alpha_t), \dep(\alpha_r))$). Moreover, $r(\alpha_t)=3123213(313(\alpha_2))=312323(\alpha_2)=313232(\alpha_2)=312(\alpha_2)=-31(\alpha_2)$, hence we have $31(\alpha_2)\in \Phi_{W'}^+\cap \Phi(p_r)=\Phi_{W'}(p_r)$. Indeed, we have $\Phi_{W'}(p_r)=\{31(\alpha_2), \alpha_r\}$, where $\dep(31(\alpha_2))=2$ and $\dep(\alpha_r)=3$. Now we are in the situation of Theorem \ref{prop:algo}(2): since $\dep(31(\alpha_2))<\dep(\alpha_r)$, then $\Delta_{W'}=\{31(\alpha_2), 31213(\alpha_r)\}=\{\alpha_1, 31(\alpha_2)\}$, and $\chi(W')=\{1, 31213\}$.
\end{Example}

\subsection*{On computing the canonical generators without a geometric representation}

In the SageMath~\cite{sage} program we wrote, we used only the abstract definition of $(W,S)$ together with its set of reflections $T$. The {\em $T$-inversion set} of $w\in W$ is well-known to be:
$$
T(w)=\{r \in T\mid \ell(rw)<\ell(w)\}\quad \text{and}\quad \Phi(w)=\{\alpha_r\mid r \in T(w) \}.
$$

Instead of considering the dominance order on positive roots, Proposition~\ref{prop:PrefDom} allows to consider dominance order on the set of reflections $T$. Given $t \in T$: 
\begin{enumerate}
    \item we compute a reduced word for $t$ and take the corresponding $t$-prefix $p_t$;
    \item we determine the $T$-inversion set $T(p_t)$ of $p_t$;  
    \item we construct the {\em $T$-dominance set} via:
$$
\dom_T(t)=\{t\}\sqcup \{ r\in T(p_t) \textrm{ such that }|\langle r,t\rangle| =\infty\}.
$$
\end{enumerate} 

Now let  $r,t\in T$ with $r\not = t$ and let $W'=\langle r,t\rangle$ be the dihedral reflection subgroup generated by $r$ and $t$. Let $p_r$ be a $r$-prefix and $p_t$ be a $t$-prefix. Then the conditions in Theorem~\ref{prop:algo} are checked on the sets $T(p_r)\cap W'$ and $T(p_t)\cap W'$ and $\dep(\cdot)$ computed on the roots is replaced by $\ell(\cdot)$ computed on the reflections in $T(p_r)\cap W'$ or $T(p_t)\cap W'$.

\begin{Remark}[On maximal dihedral reflection subgroups] For computing the canonical generators of the {\em maximal dihedral reflection subgroup} $W'$ containing $t,r$, we need a geometric representation to insure maximality, that is, $$
\Phi_{W'}=\Phi\cap(\mathbb R\alpha_r\oplus\mathbb R\alpha_t).
$$
Then we use the algorithm described above with 
$$
T_{W'}(p_r)=\{u\mid \alpha_u\in \Phi_{W'}(p_r)\}\ \text{ and }\  T_{W'}(p_t)=\{u\mid \alpha_u\in \Phi_{W'}(p_t)\}.
$$ 
where  $\Phi_{W'}(p_r)=\Phi(p_r)\cap \Phi_{W'}=\Phi(p_r)\cap(\mathbb R\alpha_r\oplus\mathbb R\alpha_t)$ and 
$$
\Phi_{W'}(p_t)=\Phi(p_t)\cap \Phi_{W'}=\Phi(p_t)\cap(\mathbb R\alpha_r\oplus\mathbb R\alpha_t).
$$
\end{Remark}


\subsection{Recognizing the language of reflection-prefixes}\label{ss:Automaton}
To conclude this section, we describe an automaton having as states the set $L_m$ of $m$-low elements (see \S\ref{ss:mSmallRoots}). This automaton recognizes the language $\Pref_T$ of reduced words for reflection-prefixes:
\begin{eqnarray*}
\Pref_T&=&\{s_1\cdots s_k s_{k+1}\in \Red\mid p=s_1\cdots s_ks_{k+1}\textrm{ is a reflection-prefix} \}\\
&=&\{s_1\cdots s_k s_{k+1}\in \Red\mid \dep(s_1\cdots s_k(\alpha_{s_{k+1}}))=k  \}\\
&=&\{s_1\cdots s_k s_{k+1}\in \Red\mid \ell( s_1\cdots  s_k s_{k+1} s_k\cdots  s_1)=2k+1\}\\
&=&\{s_1\cdots s_k s_{k+1}\in \Red\mid s_1\cdots s_k s_{k+1} s_k\cdots s_1\in \Red\}.
\end{eqnarray*}

This automaton is a directed graph on the vertex set $L_m$ with edges labeled by elements of $S$, such that for any $w \in L_m$ and $s \in S$, there is at most one edge (\textit{transition}) with source $w$ and label $s$ (see \cite[\S 2]{HNW}).

Let $\mathcal{A}:=\mathcal{A}_m(W,S)$ be the automaton such that:
\begin{itemize}
\item the set of states is $L_m$;
\item the final states are $\mathcal{F}=\{su \mid \ell(su)>\ell(u) \mbox{ and } s(\Sigma_m(u))\cap \Sigma_m (u) = \emptyset\}$;
\item the initial state is the identity $e$;
\item the transitions are: $a\in \mathcal{F} \xrightarrow{s} \pi(sa)$, where $\ell(sa)>\ell(a)$ and $\pi_m: W\rightarrow L_m$ is the Garside projection to $L_m$ defined by $\pi_m(w)= \bigvee_R \{g\in L_m\mid g\leq_R w\}$ (see \S\ref{ss:mSmallRoots}).
\end{itemize}

\begin{Example}\label{ex:auta2} 
Consider the Coxeter system $(W,S)$ of type $\widetilde{A}_2$ with Coxeter graph
\begin{center}
\begin{tikzpicture}
	[scale=2,
	 q/.style={teal,line join=round},
	 racine/.style={blue},
	 racinesimple/.style={blue},
	 racinedih/.style={blue},
	 sommet/.style={inner sep=2pt,circle,draw=black,fill=blue!40,thick,anchor=base},
	 rotate=0]
 \tikzstyle{every node}=[font=\small]
\def\grosseursimple{0.025}
\coordinate (ancre) at (0,3);

\node[sommet,label=above:$1$] (a2) at ($(ancre)+(0.25,0.4)$) {};
\node[sommet,label=below right :$3$] (a3) at ($(ancre)+(0.5,0)$) {} edge[thick] node[auto,swap,right] {}(a2) ;
\node[sommet,label=below left:$2$] (a4) at (ancre) {} edge[thick] node[auto,swap,below] {} (a3) edge[thick] node[auto,swap,left] {} (a2);
\end{tikzpicture}
\end{center}

\begin{figure}[ht!]
    \centering
    \includegraphics[width=0.7\linewidth]{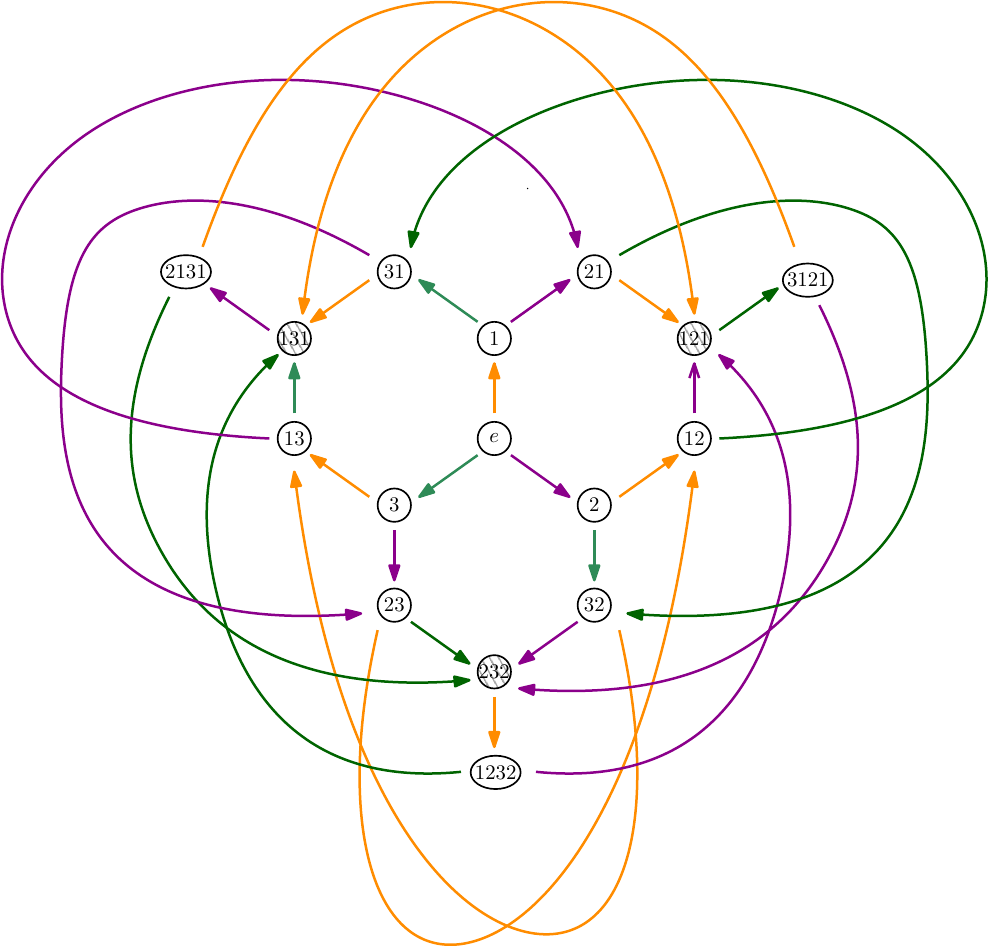}
    \caption{The automaton $\mathcal{A}_0(W,S)$ of Example \ref{ex:auta2} for type $\widetilde{A}_2$ which recognizes the language of reflection-prefixes of $W$.}
    \label{fig:aut pal}
\end{figure}

The 0-low elements of $W$ are $$L_0=\{e, 1, 2, 3, 12, 21, 13, 31, 23, 32, 121, 131, 232, 1232, 2313, 3121\}$$ (see \cite[Example 3.29]{DH}), and the automaton $\mathcal{A}_0(W,S)$ is depicted in Figure \ref{fig:aut pal}, where the shaded states $131, 121, 232$ are not final states. In fact, for instance, $\Sigma_0(21)=\{\alpha_2, \alpha_1+\alpha_2\}=s_1(\Sigma_0(21))$, hence $s_1(\Sigma_0(21))\cap\Sigma_0(21)\neq \emptyset$, so $121\notin \mathcal{F}$. Note that if the shaded states were final, then we would have obtained the Garside shadow automaton on $L_0$. 
\end{Example}

Recall that the set of left descents of $w\in W$ is $D_L(w)=\{s\in S\mid \ell(sw)<\ell(w)\}$. This first lemma is a consequence of results obtained  in~\cite{HNW} by the second author, Nadeau and Williams. 

\begin{Lemma}\label{lem:1}
Let $s_1\cdots s_{k+1}\in\Pref_T$. Then 
$$
D_L(\pi_m(s_{k+1}\cdots s_1))=D_L(s_{k+1}\cdots s_1)=\{s_{k+1}\}.
$$
\end{Lemma}
\begin{proof}
Write $u= s_{k+1}\cdots s_1$. By \cite[p.438, proof of Theorem~1.2]{HNW}, The word $s_{1} \cdots s_{k+1}$ is recognized in the state $\pi_m(u)$ and $u^{-1}$ is a reflection-prefix. By Proposition~\ref{prop:Pref1}, $D_L(u)=D_R(u^{-1})=\{s_{k+1}\}$. Therefore, by \cite[Proposition~2.6]{HNW}, we have $D_L(\pi_m(u))=D_L(u)=\{s_{k+1}\}$.
\end{proof}

\begin{Lemma}\label{lem:2}
Let $s_1\cdots s_{k+1}\in \Red$, then $s_1\cdots s_{k+1}\in\Pref_T$ if and only if $\pi_m(s_{k+1}\cdots s_{1})$ is a final state.
\end{Lemma}

\begin{proof}
By Proposition~\ref{prop:mSmall}, $\Sigma_m(\pi_m(u))=\Sigma_m(u)$ for any $u\in W$. By Lemma \cite[Lemma 2.4]{DM}, $$\dep(s_1\cdots s_{k}(\alpha_{k+1}))=k \ \mbox{if and only if} \ s_{k+1}(\Sigma_m(s_k\cdots s_1))\cap \Sigma_m(s_k\cdots s_1)=\emptyset,$$ hence $s_{k+1}(\Sigma_m(\pi_m(s_k\cdots s_1)))\cap \Sigma_m(\pi_m(s_k\cdots s_1))=\emptyset$ if and only if $\pi_m s_{k+1}\cdots s_{1})=s_{k+1}\pi_m(s_{k}\cdots s_{1})\in \mathcal{F}$.
\end{proof}

By Lemma \ref{lem:1} and Lemma \ref{lem:2} we proved the following result.
\begin{Theorem}
The automaton $\mathcal{A}$ recognizes the language $\Pref_T$.
\end{Theorem}
  
 As a consequence of the above theorem and Proposition~\ref{prop:PrefRoot} we obtain the following corollary.

\begin{Corollary}\label{cor:Refregular}
The language $\Pref_T$ is regular and the generating function 
\begin{align*}
\sum_{u\in \Pref_T} q^{|u|}& =\sum_{t\in T} |\{s_1 \cdots s_k s_{k+1} s_k\cdots s_1 \in \Pref_T \mid t=s_1 \cdots s_k s_{k+1}s_k\cdots s_1\}| \, q^{\ell(t)}\\
&=\sum_{k\in\mathbb{N}} a_k q^k
\end{align*}
is rational, where $a_k$ is the number of saturated chains of length $k$ in the root poset of $W$. In particular, $\RPal(q)=q\Pref_T(q^2)$ is rational.
\end{Corollary}

\section{Poincare series of affine Coxeter groups}\label{se:affine}
In this section we provide a direct solution to Stembridge’s problem by providing an explicit description of the Poincaré series for affine Coxeter groups. This formula is based on a detailed description of the Hasse diagram of the associated root poset.

\subsection{Crystallographic root systems}

When one works with an affine Coxeter system $(W,S)$, with underlying finite Weyl group $(W_0,S_0)$, it is customary to use its associated crystallographic root system in a positive semi-definite quadratic space $(V,B)$. In this case, the classical geometric representation (see \S\ref{ss:InversionSets}) is the unique geometric representation, and  the simple system $\Delta$ is a basis of $V$. The crystallographic condition allows us to consider root systems with non-necessarily unitary roots. We briefly recall this construction below.

Let $\Phi_0$ be a reduced, irreducible, crystallographic root system for the finite
Weyl group $W_0$, realized in a Euclidean space $(V_0,B_0)$ with corresponding set of simple roots $\Delta_0=\{\alpha_s\mid s\in S_0\}$ and positive roots $\Phi_0^+$ (see for instance \cite{bour,Hum}). Then for $s,t\in S_0$, we have:
\[
B_0(\alpha_s, \alpha_t) = 
\begin{cases} 
-\|\alpha_s\| \|\alpha_t\| \cos \left(\frac{\pi}{m_{st}}\right) & \text{if } m_{st} < \infty, \\
-1 & \text{if } m_{st} = \infty.
\end{cases}
\]
The quadratic space $(V,B)$ is obtained as follows:  $V = V_0 \oplus \mathbb{R}\delta$, and $B$ is the extension of  $B_0$ by setting the radical of $B$ to be equal to  $Q = \mathbb{R}\delta$. Let $\omega$ denote the highest root of $\Phi_0$; the crystallographic affine simple system is $\Delta = \Delta_0 \cup \{ \delta - \omega \}$ with root system 
$\Phi = \{ \alpha + k\delta \mid \alpha \in \Phi_0,\; k \in \mathbb{Z} \}$.
The set of positive roots is:
\[\Phi^+=
\{\,  k\delta + \alpha,\;  (k+1)\delta -\alpha \mid \alpha \in \Phi_0^+,\; k \in \mathbb{N} \,\}=(\mathbb{N}\delta+\Phi^+_0)\sqcup (\mathbb{N}^*\delta-\Phi^+_0).
\]
It follows that $\Phi$ is a {\em crystallographic (infinite) root system} for the affine Weyl group $(W,S)$, where $S=S_0 \cup \{s_{\delta-\omega}\}$.

\subsection{The root poset and dominance order in the affine case} Structural results for unitary root systems, specifically those determined by the sign of the bilinear form $B$, generally extend to crystallographic systems. However, properties that depend on the specific numerical values of the form, such as dominance and long edges, must be precisely reformulated for crystallographic root systems; see~\cite[Example 3.9]{DH} for more details.

From the crystallographic root system, $\Phi$, one obtains a {\em unitary} root system in the sense of \S\ref{ss:InversionSets} by normalizing each root via the map $\alpha\mapsto \frac{\alpha}{\|\alpha\|}$. Let $\alpha,\beta\in\Phi$, then:     
$$
B\left(\frac{\alpha}{\|\alpha\|},\frac{\beta}{\|\beta\|}\right) = \frac{B(\alpha,\beta)}{\|\alpha\|\|\beta\|}.
$$
Therefore, the condition $B(\alpha',\beta')\leq -1$ for unitary roots $\alpha',\beta'$, which appear in Definition~\ref{def:long} or Proposition~\ref{prop:inf-depth}, becomes:  $$B(\alpha,\beta)\leq -\|\alpha\|\|\beta\|.$$ 
For this reason Definition~\ref{def:long} in the crystallographic setting becomes the following.

\begin{Definition}[Long and short edges in crystallographic root system] A covering $\alpha\lhd s(\alpha)$ in the root poset $(\Phi^+,\leq)$ is \textit{short} if $-\|\alpha\| \|\alpha_s\| < B(\alpha, \alpha_s)<0$. Otherwise we say that the covering is \textit{long}.
\end{Definition}

The Hasse diagram of the root poset is studied in more details in the subsequent section. 

\smallskip
In the affine case, the Hasse diagram of the dominance order is the union of $2|\Phi_0^+|$ total ordered chains. Indeed, the only dominance relations in  $\Phi$ are of the form
$$
 k\delta +\alpha \;\leq_d\;  l\delta+\alpha,
\quad
\text{or} \quad 
 (k+1)\delta -\alpha\;\leq_d\;  (l+1)\delta - \alpha,
$$
where $k\leq l$ in $\mathbb N$ and $\alpha\in\Phi_0^+$; see~\cite[Example 3.9]{DH}. 

An interesting consequence of these dominance relations is that for $\alpha\in\Phi_0^+$ and $k\in\mathbb N$ we have:
$$
\dep_\infty(k\delta+\alpha)=k\quad\textrm{and}\quad\dep_\infty((k+1)\delta-\alpha)=k,
$$
see~\cite[Example 3.9, Eq.($\diamond$)]{DH}. In particular, the set of $k$-small roots (see Section \ref{ss:mSmallRoots}) for $k\in\mathbb N$ is in the affine case:
\begin{equation}
\label{eq:mSmall}
\Sigma_k=\left(k\delta+\Phi^+_0\right)\sqcup \left((k+1)\delta-\Phi_0^+\right).
\end{equation}

\begin{Example}
Let $(W,S)$ be a Coxeter system of type $\widetilde{A}_2$, 

\begin{center}
\begin{tikzpicture}
	[scale=2,
	 q/.style={teal,line join=round},
	 racine/.style={blue},
	 racinesimple/.style={blue},
	 racinedih/.style={blue},
	 sommet/.style={inner sep=2pt,circle,draw=black,fill=blue!40,thick,anchor=base},
	 rotate=0]
 \tikzstyle{every node}=[font=\small]
\def\grosseursimple{0.025}
\coordinate (ancre) at (0,3);

\node[sommet,label=above:$s_{\delta-\omega}$] (a2) at ($(ancre)+(0.25,0.4)$) {};
\node[sommet,label=below right : $s_2$] (a3) at ($(ancre)+(0.5,0)$) {} edge[thick] node[auto,swap,right] {}(a2) ;
\node[sommet,label=below left:$s_1$] (a4) at (ancre) {} edge[thick] node[auto,swap,below] {} (a3) edge[thick] node[auto,swap,left] {} (a2);
\end{tikzpicture}
\end{center}
\noindent
then the underlying finite Coxeter system $(W_0,S_0)$ is of type $A_2$. Consider $S=\{s_1, s_2, s_{\delta-\omega}\}$, where $s_1,s_2\in S_0$ and $\omega=\alpha_1+\alpha_2$ is the highest root of $W_0$. The finite positive root system is $\Phi^+_0=\{\alpha_1, \alpha_2, \alpha_1+\alpha_2\}$ and the affine simple system is $\Delta=\{\alpha_1, \alpha_2, \delta-\omega\}$. The set of $k$-small roots for $k\in\mathbb N$ is:
\begin{eqnarray*}
\Sigma_k &=& \{\alpha_1 + k\delta, \alpha_2 + k\delta, \alpha_1+\alpha_2+k\delta\}\\
    &&\sqcup\  \{-\alpha_1+(k+1)\delta, -\alpha_2+(k+1)\delta, -(\alpha_1+\alpha_2)+(k+1)\delta\}. 
\end{eqnarray*}
The affine positive root system is:
$$
\Phi^+=\bigsqcup_{m\in\mathbb N} \Sigma_m.
$$
The Hasse diagram of $(\Phi^+,\leq)$ is depicted in Figure \ref{fig:rootposetA2}, up to depth equal to 3, where each color of the edges corresponds to the simple reflection that is applied: red for $s_1$, yellow for $s_2$ and green for $s_3:=s_{\delta -\omega}$. Moreover, the dashed edges correspond to the long coverings.
\end{Example}

\begin{figure}[ht]
    \centering    \includegraphics[width=0.9\linewidth]{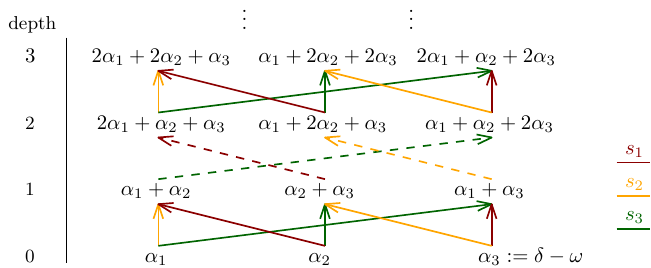}
    \caption{Hasse diagram of the root poset of type $\widetilde{A}_2$ up to depth~3, with $\Sigma=\Sigma_0$ and $\Sigma_1$ that correspond to the blocks connected with non-dotted edges.}
    \label{fig:rootposetA2}
\end{figure}


\begin{Example}
    Let $(W,S)$ be a Coxeter system of type $\widetilde{B}_3$ with the following Coxeter graph
    \begin{center}
\begin{tikzpicture}
	[scale=2,
	 q/.style={teal,line join=round},
	 racine/.style={blue},
	 racinesimple/.style={blue},
	 racinedih/.style={blue},
	 sommet/.style={inner sep=2pt,circle,draw=black,fill=blue!40,thick,anchor=base},
	 rotate=0]
 \tikzstyle{every node}=[font=\small]
\def\grosseursimple{0.025}
\coordinate (ancre) at (0,3);

\node[sommet,label=right:$s_3$] (a2) at ($(ancre)+(0.8,0.25)$) {};
\node[sommet,label=above :$s_2$] (a3) at ($(ancre)+(0.5,0)$) {} edge[thick] node[auto,swap,right] {}(a2) ;
\node[sommet,label=above:$s_1$] (a4) at (ancre) {} edge[thick] node[auto,swap,below] {$4$} (a3);
\node[sommet,label=right :$s_{\delta-\omega}$] (a5) at ($(ancre)+(0.8,-0.25)$) {} edge[thick] node[auto,swap,right] {}(a3) ;
\end{tikzpicture}
\end{center}

\noindent    
    In this case we have $\|\alpha_1\|=1$ and $\|\alpha_2\|=\|\alpha_3\|=\|\omega\|=\sqrt{2}$, where $\omega=2\alpha_1+2\alpha_2+\alpha_3$. For instance, $B(\alpha_2,\delta-\omega)=-B(\alpha_2,\omega)=-1>-\|\alpha_2\|\|\omega\|=-2$, and $B(\delta-\alpha_3, \alpha_3)=-\|\alpha_3\|^2=-2$. Hence the covering $\alpha_2\lhd s_{\delta-\omega}(\alpha_2)$ is short, while $\delta-\alpha_3\lhd s_3(\delta-\alpha_3)=\delta+\alpha_3$ is long. The Hasse diagram of the root poset  restricted to $\Sigma=\Sigma_0$ of type $\widetilde{B}_3$ is in Figure \ref{fig:root-poset-b3}.
\end{Example}

\subsection{Poset isomorphisms within the root poset}

Let $(W,S)$ be an affine Coxeter system with underlying finite system $(W_0,S_0)$ and crystallographic root system $\Phi_0$. Since $\Phi_0$ is irreducible, we have two cases for the decomposition of $\Phi_0$ into $W_0$-orbits:
\begin{itemize}
    \item $\Phi_0=\mo_1$ is a single orbit if $(W,S)$ is simply laced;
    \item $\Phi_0=\mo_1\sqcup\mo_2$ if $(W,S)$ is of type $\tilde B_n$, $\tilde C_n$, $\tilde G_2$ or $\tilde F_4$.
\end{itemize}
For simplification, we write:
$$
\Phi_0^+=\bigsqcup_{i=1}^h \mo_i^+,
$$ 
where $\mo_i^+=\mo_i \cap \Phi_0^+$ and $1\leq i\leq h$ ($h=1,2$). We also assume that $\mo_1^+$ is the $W_0$-orbit containing the highest root $\omega$. 

\smallskip

Let $\mo$ be a $W_0$-orbit. For $k\in\mathbb N$, the set of $k$-small roots in $\mo$ is:
\begin{equation}
\Sigma_{\mo,k} =(k\delta+\mo^+)\sqcup \left((k+1)\delta-\mo^+\right) =\{k\delta + \alpha, (k+1)\delta - \alpha \mid \alpha \in \mo^+\}.
\end{equation}
Moreover, by Eq~(\ref{eq:mSmall}), we have $\Sigma_k=\sqcup_{i=1}^h\Sigma_{\mo_i,k}$, where $h=1,2$. Set 
$$
\Phi_{\mo}=\bigsqcup_{k\in \mathbb N} \Sigma_{\mo,k}.
$$
So $\Phi^+=\sqcup_{i=1}^h\Phi_{\mo_i}$, where $h=1,2$.

The Hasse diagrams of the subposet $(\Phi_{\mo},\leq)$  exhibit strong symmetry properties that are fundamental to prove the main result of this section. 

\subsection*{Symmetry at the level of small roots} We denote for simplification, the set of small root in $\mo$ by:
$$
\Sigma_{\mo}=\Sigma_{\mo,0}=\mo^+\sqcup (\delta-\mo^+).
$$
Recall that $\Sigma$ is an order ideal in the root poset (see~\cite[Corollary~4.7.7]{BB}) and any cover relation within $\Sigma$ is short (see~\cite[Lemma~4.7.5 and Theorem~4.7.6]{BB}).

\begin{Theorem}\label{th:iso1} Let $\mo$ be a $W_0$-orbit. 
The map $\phi: \mo^+ \to \delta - \mo^+$ defined by $\phi(\alpha) = \delta - \alpha$ is an isomorphism of posets between $(\mo^+, \leq)$ and the opposite poset $(\delta - \mo^+, \leq)^{op}$: for any $\alpha, \beta \in \mo^+$, we have $\alpha \leq \beta$ if and only of  $\delta - \beta \leq \delta - \alpha$. Moreover all covering relations in $(\mo^+,\leq)$ and in $(\delta - \mo^+, \leq)$ are short.
\end{Theorem}

\begin{proof}
Since a poset isomorphism is uniquely determined by the preservation (or reversal) of covering relations, it suffices to show that $\alpha \lhd \beta$ in $\mo^+$ if and only if $\delta - \beta \lhd \delta - \alpha$ in $\delta - \mo^+$. Recall that for any $\alpha, \beta \in \Phi^+$, a covering $\alpha \lhd \beta$ exists if and only if there is a simple reflection $s \in S$ such that $\beta = s(\alpha)$ and $B(\alpha, \alpha_s) < 0$. We examine the two types of simple reflections in the affine Weyl group $W$.

\begin{itemize}
    \item[(i)] Let $s \in S_0$ be a finite reflection. For a finite simple root $\alpha_s$, we have:
    \[ B(\delta - \alpha, \alpha_s) = B(\delta, \alpha_s) - B(\alpha, \alpha_s) = -B(\alpha, \alpha_s). \]
    It follows that $B(\alpha, \alpha_s) < 0$ if and only if $B(\delta - \alpha, \alpha_s) > 0$. Thus, $\alpha \lhd s(\alpha)$ if and only if  $s(\delta - \alpha) = \delta - s(\alpha) \lhd \delta - \alpha$.

    \item[(ii)] Let $s_{\delta-\omega}$ be the affine reflection. Similarly, for the affine simple root $\delta - \omega$, the property of the radical implies:
    \[ B(\delta - \alpha, \delta - \omega) = B(\delta, \delta - \omega) - B(\alpha, \delta - \omega) = -B(\alpha, \delta - \omega). \]
    Again, the sign of the bilinear form is reversed: $B(\alpha, \delta - \omega) < 0$ if and only if $B(\delta - \alpha, \delta - \omega) > 0$. This ensures that $\alpha \lhd s_{\delta-\omega}(\alpha)$ if and only if $s_{\delta-\omega}(\delta - \alpha) = \delta - s_{\delta-\omega}(\alpha) \lhd \delta - \alpha$.
\end{itemize}

Now, let $\alpha \lhd s(\alpha)$ be a covering in $\mo^+$ for some $s \in S$. The type of this cover depends on the value $B(\alpha, \alpha_s)$. Under the map $\phi$, this corresponds to the cover $\delta-s(\alpha) \lhd \delta -\alpha$. By the properties of the reflection $s$ and of the radical, we have:
\[
B(\delta - s(\alpha), \alpha_s) = -B(s(\alpha), \alpha_s) = B(\alpha, \alpha_s).
\]
Since  $B(\cdot, \alpha_s)$ is identical in both case, the condition for being short or long in Definition \ref{def:long} is preserved. Since $\mo^+$ and $\delta-\mo^+$ decomposes the set of small roots~$\Sigma$, so they are all short. 
\end{proof}

\begin{Example}
    If $(W,S)$ is of type $\widetilde{B}_3$, then $\Phi_0^+=\mo_1^+\sqcup\mo_2^+$, where $\mo_1^+=\{\alpha_1, s_2(\alpha_1), s_3s_2(\alpha_1)\}$ and $\mo_2^+=\{\alpha_2, \alpha_3, s_1(\alpha_2), s_3(\alpha_2), s_1s_3(\alpha_2), s_2s_1s_3(\alpha_2)\}$. The root poset restricted to $\Sigma_{\mo_1}\sqcup\Sigma_{\mo_2}$ is depicted in Figure \ref{fig:root-poset-b3}.
\end{Example}

\begin{figure}[hbtp]
\centering
\includegraphics[scale=.75]{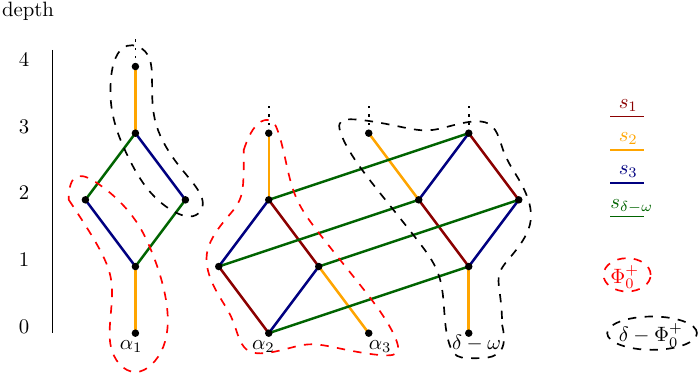}
\caption{Root poset of type $\widetilde{B}_3$ restricted to $\Sigma_{\mo_1}\sqcup \Sigma_{\mo_2}$. The portion of poset in the red lines is isomorphic to opposite in the black lines, as stated in Theorem \ref{th:iso1}.}
\label{fig:root-poset-b3}
\end{figure}

\subsection*{Symmetry of the root poset restricted to a $W_0$- orbit}
Following \cite[Ch. VI, \S 1]{bour}, for $\alpha, \beta\in \Phi^+$, we set 
$$n(\alpha,\beta):=2\frac{B(\alpha,\beta)}{B(\beta,\beta)},$$  
so that $s_\beta(\alpha)=\alpha - n(\alpha,\beta)\beta.$ By \cite[Ch. VI, \S 1, Proposition~8]{bour}, if $\alpha\neq \beta$ and $\|\alpha\|\leq \|\beta\|$, then $n(\alpha,\beta) \in \{0,\pm 1\}$. It follows that $n(\alpha, \omega) \in \{0,\pm 1\}$ when $\alpha\neq \omega$, since $\|\alpha\|\leq \|\omega\|$ for all positive roots $\alpha$. Finally, note that $n(k\delta+\alpha, \beta)=n(\alpha, \beta)=n(\alpha, m\delta+\beta)$ and $n(k\delta-\alpha, \beta)=-n(\alpha, \beta)=n(\alpha, m\delta-\beta)$ for $k,m\in\mathbb{N}^*$.
\smallskip

 We describe now the covering relations inside the sets $\Sigma_{\mo,k}$. Let $\alpha \in \Phi_0^+$ and $k \in \mathbb{N}$.

\begin{enumerate}
\item[(1)]
There are no covering relations of the form
\[
(k+1)\delta - \alpha \lhd s_{\delta-\omega}\big((k+1)\delta - \alpha\big).
\]
Indeed, if such a covering existed, then
$0 > B\big((k+1)\delta - \alpha,\, \delta - \omega\big) = B(\alpha,\omega)$,
so that $n(\alpha,\omega) = -1$. It follows that
\[
s_{\delta-\omega}(\alpha)
= \alpha - n(\alpha,\delta-\omega)(\delta-\omega)
= \alpha + n(\alpha,\omega)(\delta-\omega)
= \alpha - \delta + \omega.
\]
Thus, $s_{\delta-\omega}\big((k+1)\delta - \alpha\big)
= (k+2)\delta - (\alpha + \omega)
\notin \Sigma_{\mo,k},$ a contradiction.
\smallskip

\item[(2)]
All covering relations in $\Sigma_{\mo,k}$ are of one of the following three types:
\begin{itemize}
\item[(i)]
$k\delta + \alpha \lhd s(k\delta + \alpha) = k\delta + s(\alpha),
\qquad \text{when } B(\alpha,\alpha_s) < 0;
$

\item[(ii)]
$
(k+1)\delta - \alpha \lhd s\big((k+1)\delta - \alpha\big)
= (k+1)\delta - s(\alpha),
$
when $B(\alpha,\alpha_s) > 0$ and $\alpha \neq \alpha_s$.
Note that if $\alpha = \alpha_s$, then
$(k+1)\delta - s(\alpha) = (k+1)\delta + \alpha_s \in \Sigma_{\mo,k+1}$.

\item[(iii)]
$
k\delta + \alpha \lhd s_{\delta-\omega}(k\delta + \alpha)
= k\delta + s_{\delta-\omega}(\alpha),
$
when $B(\alpha,\delta-\omega) < 0$ and $\alpha \neq \omega$.
Otherwise, $s_{\delta-\omega}(k\delta + \omega) = (k+2)\delta - \omega
\in \Sigma_{\mo,k+1}.$
\end{itemize}
\end{enumerate}

We state now the second symmetry of the root poset.

\begin{Theorem}\label{th:iso2} Let $\mo$ be a $W_0$-orbit. There is a poset isomorphism $$(\Sigma_{\mo,k}, \leq) \cong (\Sigma_{\mo}, \leq)$$ for all $k\in\mathbb{N}$. Moreover, all covering in $\Sigma_{\mo,k}$ are short.
\end{Theorem}

\begin{proof}
The map $\pi_k : \Sigma_{\mo,k} \longrightarrow \Sigma_{\mo}$ defined by
\[
\pi_k(k\delta+\alpha)=\alpha,
\qquad
\pi_k((k+1)\delta-\alpha)=\delta-\alpha,
\]
for $\alpha\in\Phi_0^+$ is a bijection. By the discussion preceding the theorem, we have three types of coverings in $\Sigma_{\mo,k}$.
\smallskip

A covering $k\delta+\alpha \lhd k\delta+s(\alpha)$ of type (i) occurs if and only if
$0>B(k\delta+\alpha,\alpha_s)=B(\alpha,\alpha_s)$,
which is equivalent to $\alpha \lhd s(\alpha)$ being a covering in $\Sigma_{\mo}$. 
\smallskip

A covering $(k+1)\delta-\alpha \lhd (k+1)\delta-s(\alpha)$ of type (ii), occurs if and only if
$ 0>B((k+1)\delta-\alpha,\alpha_s)=-B(\alpha,\alpha_s)$,
with $\alpha\neq\alpha_s$, which is equivalent to
$\delta-\alpha \lhd \delta-s(\alpha)$
being a covering in $\Sigma_{\mo}$.
\smallskip

A covering $k\delta+\alpha \lhd s_{\delta-\omega}(k\delta+\alpha)$ of type (iii)
occurs if and only if $0>B(k\delta+\alpha,\delta-\omega)=B(\alpha,\delta-\omega)$, with $\alpha\neq\omega$, which is equivalent to $\alpha \lhd s_{\delta-\omega}(\alpha)
$ being a covering in $\Sigma_{\mo}$. 
\smallskip

Thus $\pi_k$ preserves and reflects covering relations and is therefore a poset isomorphism. The final statement follows from Theorem~\ref{th:iso1} since all covering are short in $\Sigma_\mo$.
\end{proof}

\begin{Example} In the Figure~\ref{fig:root-poset-b3-iso2} below, the the first three slices of the Hasse diagram of the root poset of type $\widetilde{B}_3$ are represented. The first slice coincides with the poset shown in Figure \ref{fig:root-poset-b3}, although the order relations are drawn differently. This alternative depiction makes the isomorphisms described in Theorem \ref{th:iso2} more apparent.
\end{Example}

\begin{figure}[hbtp]
\centering
\includegraphics[scale=.7]{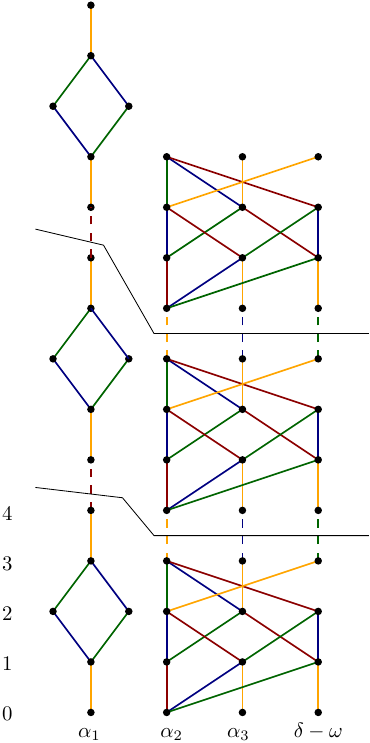}
\caption{}
\label{fig:root-poset-b3-iso2}
\end{figure}

We end this discussion by showing that a covering $\alpha\lhd \beta$ with $\alpha\in \Sigma_{\mo,k}$ and $\beta\in \Sigma_{\mo,k+1}$ is long.

\begin{Proposition}[Covering relations]\label{prop:long-cover}\
\begin{itemize}
    \item[(i)]  For all $k\in\mathbb{N}, s\in S_0$, $(k+1)\delta-\alpha_s\lhd s((k+1)\delta - \alpha_s)$ is a long covering, where $(k+1)\delta-\alpha_s\in \Sigma_{\mo,k}$ and $s((k+1)\delta - \alpha_s)\in \Sigma_{\mo, k+1}$.
    \item[(ii)] For all $k\in\mathbb{N}$, $k\delta +\omega\lhd s_{\delta-\omega}(k\delta + \omega)$ is a long covering, where $k\delta +\omega\in \Sigma_{\mo,k}$ and $s_{\delta-\omega}(k\delta + \omega)\in \Sigma_{\mo,k+1}$. 
\end{itemize}
   
\end{Proposition}

\begin{proof}
    Consider the  $(k+1)\delta-\alpha_s\lhd s((k+1)\delta - \alpha_s)$, then $B((k+1)\delta-\alpha_s, \alpha_s)=-\|\alpha_s\|^2$, hence the covering is long, and $s((k+1)\delta - \alpha_s)=(k+1)\delta+\alpha_s \in \Sigma_{\mo,k+1}$. Analogously, if we consider the covering $k\delta +\omega\lhd s_{\delta-\omega}(k\delta + \omega)$, $B(k\delta +\omega, \delta-\omega)=-\|\omega\|^2$, so the covering is long. In particular, $s_{\delta -\omega}(k\delta +\omega)=(k+2)\delta -\omega \in \Sigma_{\mo,k+1}$. 
    \end{proof}

\subsection{An explicit formula for the generating function of depth} We are ready now to provide a closed formula for the generating function of the depth of positive roots in $\Phi$, using the restriction of the root poset to the sets $\Sigma_\mo$ and $\Phi_\mo$ introduced in the previous section.
\medskip

Let $\mo \subseteq \Phi_0$ be $W_0$-orbit. Define 
$$
M_{\mo}=\max\{\dep(\beta) \mid \beta \in \Sigma_{\mo}\}.
$$
We show that the depth of a positive root depends only on $M_{\mo}$ and on
the depth of a root in $\mo^+$.

\begin{Corollary}\label{cor:maxdep}
    Let $\alpha\in \mo^+$ and $k\in \mathbb{N}$, then :
    \begin{enumerate}
        \item $\dep(k\delta + \alpha)=\dep(\alpha)+k(M_\mo +1)$;
        \item $\dep((k+1)\delta -\alpha) = \dep(\delta-\alpha) +k(M_\mo +1)$.
    \end{enumerate}  
\end{Corollary}
\begin{proof}
    By Theorem \ref{th:iso2}, the lowest depth in $\Sigma_{\mo, k}$ is $k(M_\mo +1)$.
    It suffices to show that the only roots in $\Sigma_{\mo,k}$ that have depth equal to $k(M_\mo +1)$ are $$\{k\delta + \alpha_s, (k+1)\delta - \omega \mid s\in S\}.$$ 
    Note that $B(k\delta+\alpha_s, \alpha_s)=B(\alpha_s,\alpha_s)>0$, and $s(k\delta+\alpha_s)=k\delta-\alpha_s\in \Sigma_{\mo, k-1}$. So we have the covering $k\delta-\alpha_s\lhd k\delta+\alpha_s$. Analogously, we also have the covering $k\delta-\omega \lhd (k+1)\delta -\omega$, where $k\delta-\omega\in \Sigma_{\mo,k-1}$. Hence $\{k\delta + \alpha_s, (k+1)\delta - \omega \mid s\in S\}$ are the roots with lowest depth in $\Sigma_{\mo,k}$ and there are no more roots with this property.
\end{proof}

We denote the generating functions of the depth over the whole set $\Phi^+$, and its restriction to $\Sigma_\mo$ and $\Phi_\mo$ respectively, by: 
\begin{equation}\label{def:depth-polys}
\Phi^+(q)=\sum_{\beta\in \Phi^+} q^{\dep(\beta)}, \quad  P_{\mo}(q):=\sum_{\beta\in \Sigma_\mo} q^{\dep(\beta)},\quad \mbox{ and }\quad \Phi_{\mo}(q):=\sum_{\beta\in \Phi_{\mo}} q^{\dep(\beta)}.
\end{equation} 

\begin{Proposition} \label{prop:palindrom}
The polynomial $P_{\mo}(q)$ is palindromic of degree $M_\mo$, and the generating function for the depth of roots in $\Phi_{\mo}$ is given by:
\begin{equation}
    \Phi_{\mo}(q) = \frac{P_{\mo}(q)}{1 - q^{M_{\mo}+1}}.
\end{equation}
\end{Proposition} 
\begin{proof}
First of all we observe that by the definitions of $\Sigma_{\mo}$ and $M_{\mo}$, the maximum degree of the terms in the sum defining $P_{\mo}(q)$ is exactly $M_{\mo}$. More precisely, we have
$$P_{\mo}(q)=\sum_{\beta\in \mo^+} (q^{\dep(\beta)}+q^{M_\mo-\dep(\beta)}).$$ 
Since 
\begin{align*}
q^{M_\mo}P_\mo(q^{-1}) &= q^{M_\mo} \sum_{\beta\in \mo^+} (q^{-\dep(\beta)} + q^{-(M_\mo-\dep(\beta))}) \\
&= \sum_{\beta\in \mo^+} (q^{M_\mo-\dep(\beta)} + q^{\dep(\beta)}) = P_\mo(q),
\end{align*}
the polynomial $P_{\mo}(q)$ is palindromic. 

By Corollary \ref{cor:maxdep}, the depth of each root in the $k$-th slice is shifted by $k(M_{\mo}+1)$, namely 
$$\sum_{\beta\in\Sigma_{\mo,k}} q^{\dep(\beta)}=q^{k(M_\mo +1)}P_\mo(q).$$
Summing these contributions yields the geometric series:
\[
\Phi_{\mo}(q) = \sum_{k\in\mathbb{N}}\sum_{\beta\in\Sigma_{\mo,k}} q^{\dep(\beta)}=\sum_{k=0}^{\infty} q^{k(M_{\mo}+1)} P_{\mo}(q) = \frac{P_{\mo}(q)}{1 - q^{M_{\mo}+1}},
\]
and this completes the proof.
\end{proof}

By combining the polynomials associated to each orbit, we obtain our final result.

\def\lcm{{\operatorname{lcm}}}

\begin{Theorem}\label{th:main}
    The generating function of the depth of positive roots is of the form
    $$\Phi^+(q)=\frac{P(q)}{1-q^M}$$
    where $M=\lcm(M_\mo+1\mid \mo\ \mbox{ $W_0$-orbit})$ and $P$ is palindromic.
\end{Theorem}

\begin{proof}
    By construction, $\Phi^+(q)=\sum_{i=1}^{h}\Phi_{\mo_i}(q)$, where $h=1,2$ is the number of $W_0$-orbits. For simplicity, set $P_{\mo_i}=P_i$ and $M_{\mo_i}=M_i$ for all $i$, and let $M=\lcm(M_{i}+1)_{i=1}^h$. Then we may rewrite $\Phi^+(q)$ as $$\Phi^+(q)=\frac{\sum_{i=1}^k\frac{1-q^M}{1-q^{M_i+1}}P_i(q)}{1-q^M}=\frac{P(q)}{1-q^M}.$$
    Now we want to show that $P(q)$ is palindromic. Denote by $D_i(q)=\frac{1-q^M}{1-q^{M_i+1}}$ for every $i$. It is not difficult to see that $D_i(q)=\sum_{j=1}^{M/(M_i+1)}q^{M-j(M_i+1)}$ is a palindromic polynomial of $\deg(D_i(q))=M-M_i-1$. In particular, $\deg(D_i(q)P_i(q))=M-1$ for all $i$, so $\deg(P(q))=M-1$. We have 
    \begin{align*}  
    q^{M-1}P(q^{-1})=&\sum_{i=1}^k q^{M-1}D_i(q^{-1})P_i(q^{-1})\\
    =&\sum_{i=1}^k q^{M-M_i-1}D_i(q^{-1})q^{M_i}P_i(q^{-1})\\
    =&\sum_{i=1}^k D_i(q)P_i(q)=P(q)
    \end{align*}
    where the third equality holds since $D_i(q)$ and $P_i(q)$ are palindromic for all $i$ (Proposition \ref{prop:palindrom}). Therefore, $P(q)$ is palindromic.
\end{proof}

\begin{Example}\label{ex:poly-b3}
    Consider $(W,S)$ of type $\widetilde{B}_3$.  The polynomials corresponding to the two orbits are: 
    $$
    P_{\mo_1}(q)= 1+q+2q^2+q^3+q^4 \ \mbox{and} \  P_{\mo_2}(q)=3+3q+3q^2+3q^3.
    $$
    Therefore $M_{\mo_1}=4$ and $M_{\mo_2}=3$. So from Proposition~\ref{prop:palindrom} we obtain:
    $$
    \Phi_{\mo_1} =\frac{1+q+2q^2+q^3+q^4}{1-q^5}\ \mbox{and} \ \Phi_{\mo_2} =\frac{3+3q+3q^2+3q^3}{1-q^4}=\frac{3}{1-q}.
    $$
     Following the notation in the proof of Theorem~\ref{th:main}, we obtain:
     $$
     D_1(q)=1+q^5+q^{10}+q^{15}\ \text{ and }\ D_2(q)=1+q^4+q^8+q^{12}+q^{16}.
     $$
     Since   $M=\lcm(5,4)=20$, the generating function of the depth of $\Phi^+$ is: 
    \begin{align*}
        \Phi^+(q)=& (4q^{19} + 4q^{18} + 5q^{17} + 4q^{16} + 4q^{15} + 4q^{14} + 4q^{13} + 5q^{12} + 4q^{11} + 4q^{10}\\& + 4q^9 + 4q^8 + 5q^7 + 4q^6 + 4q^5 + 4q^4 + 4q^3 + 5q^2 + 4q + 4)(1-q^{20})^{-1}.
    \end{align*}  
\end{Example}

\begin{Example}
    Consider $(W,S)$ of type $\widetilde{F}_4$. We obtain the Hasse diagram of root poset restricted to $\Sigma_{\mo_1}$ and $\Sigma_{\mo_2}$ using SageMath. The polynomials corresponding to the two orbits are: 
    \begin{eqnarray*}
        &P_{\mo_1}(q)= 3\sum_{i=0}^7 q^i\quad \mbox{and} &\\
        &  P_{\mo_2}(q)=2+2q+2q^2+3q^3+2q^4+2q^5+2q^6+3q^7+2q^8+2q^9+2q^{10}.&
    \end{eqnarray*}
    Therefore $M_{\mo_1}=7$ and $M_{\mo_2}=10$. So from Proposition~\ref{prop:palindrom} we obtain:
    $$
    \Phi_{\mo_1} =\frac{3}{1-q}\quad \mbox{and} \quad \Phi_{\mo_2} =\frac{P_{\mo_2}}{1-q^{11}}.
    $$
    Note that  $M=\lcm(11,8)=88$.
\end{Example}

\begin{Example}
    Consider $(W,S)$ of type $\widetilde{G}_2$. We obtain the Hasse diagram of root poset restricted on $\Sigma_{\mo_1}$ and $\Sigma_{\mo_2}$ using SageMath. The polynomials corresponding to the two orbits are: 
    $$
        P_{\mo_1}(q)= 2+2q+2q^2\ \mbox{and} \ 
          P_{\mo_2}(q)= 1+q+2q^2+q^3+q^4
    $$
    Therefore $M_{\mo_1}=2$ and $M_{\mo_2}=4$. So from Proposition~\ref{prop:palindrom} we obtain:
    $$
    \Phi_{\mo_1} =\frac{2}{1-q}\quad \mbox{and} \quad \Phi_{\mo_2} =\frac{1+q+2q^2+q^3+q^4}{1-q^{5}}.
    $$
    Note that  $M=\lcm(3,5)=15$.
\end{Example}

\subsection*{Example of classical types}

Let $n\in\mathbb N_{\geq 2}$.

\begin{Example}
    In type $\widetilde{A}_n$, since there is only one orbit, we obtain by direct computation that $M=n-1$, which is the height of the highest root minus $1$, and  
    $$
    P_{\mo}(q)=P(q)=(n+1)\sum_{i=0}^{n-1}q^i=(n+1)\frac{1-q^{n}}{1-q}.
    $$
    Therefore:
    $$
    \Phi^+(q)=\frac{n+1}{1-q}.
    $$
\end{Example}

\begin{Example}
    Consider $(W,S)$ of type $\widetilde{B}_n$. One shows that the polynomials corresponding to the two orbits are: 
    \begin{eqnarray*}
        &P_{\mo_1}(q)= \sum_{i=0}^{n-1} q^i\quad \mbox{and} &\\
        &  P_{\mo_2}(q)=\sum_{i=0}^{n-2}(n-1-i)\left(q^{2i}+q^{2i+1}\right).&
    \end{eqnarray*}
    Therefore $M_{\mo_1}=n-1$ and $M_{\mo_2}=2n-3$. So from Proposition~\ref{prop:palindrom} we obtain:
    $$
    \Phi_{\mo_1} =\frac{1}{1-q}\quad \mbox{and} \quad \Phi_{\mo_2} =\frac{P_{\mo_2}}{1-q^{2n-2}}.
    $$
\end{Example}

\begin{Example}
    Consider $(W,S)$ of type $\widetilde{C}_n$. One shows that the polynomials corresponding to the two orbits are: 
    \begin{eqnarray*}
        &P_{\mo_1}(q)= 2\sum_{i=0}^{n-1} q^i\quad \mbox{and} &\\
        &  P_{\mo_2}(q)= (n-1)q^{2n-2}+\sum_{i=0}^{n-2}(n-1)q^{2i}+nq^{2i+1}.&
    \end{eqnarray*}
    Therefore $M_{\mo_1}=n-1$ and $M_{\mo_2}=2n-2$. So from Proposition~\ref{prop:palindrom} we obtain:
    $$
    \Phi_{\mo_1} =\frac{2}{1-q}\quad \mbox{and} \quad \Phi_{\mo_2} =\frac{P_{\mo_2}}{1-q^{2n-1}}.
    $$
\end{Example}

\section{Open problems}\label{se:openproblems}

We propose here some natural problems arising from this article.



\begin{Problem}
    Is the language of palindromic reduced word (or lexicographically ordered palindromic reduced word) regular?
\end{Problem}

\begin{Problem}
    Characterize which Grassmannian elements (elements with only one right descent)  are reflection prefixes  (see~Proposition~\ref{prop:Pref1}). (Note that Grassmannian elements are also minimal representatives in right weak order for Nathan Reading's shards~\cite{Re11-1}.)
\end{Problem}

\begin{Problem} For indefinite Coxeter systems, i.e., not affine or finite, find closed formulas for the generating function $\Phi^+(q)$ and $T(q)=q\Phi^+(q^2)$.
\end{Problem}

\subsection*{Problem 4} In the case of affine types, we provide a combinatorial description  of the Hasse diagram of the root poset and recall that the dominance order is just the union of total orders. In the case of indefinite Coxeter systems, Dyer and the second authors initiated the study of {\em weak orders on roots}, which are a generalization of both orders, see~\cite[\S5]{DH}. A first fundamental difference between both orders is that the dominance order is not graded by the infinite-depth (whereas the root poset is graded by the depth). That led us to propose the following questions in the case of indefinite Coxeter systems.
{\em \begin{enumerate}
    \item Study the Hasse diagram of the root poset and its connected compoments obtained by removing the long edges.
    \item Study the formal power series that distinguished short and long covers in the root poset: $\sum_{k,l\in\mathbb{N}} a_{k,l}q^kt^l$, where $a_{k,l}$ is the number of saturated chains in the root poset with $k$ short edges and $l$ long edges. 
    \item Study the dominance order. 
    \item Study the generating function of infinite depth $\sum_{\beta\in\Phi^+} q^{\dep_{\infty}(\beta)}$.
    \item Study the formal power series $\sum_{k\in\mathbb{N}} b_kq^k$, where $b_k$ is the number of saturated chains in the dominance order of length $k$. (This problem is related to~\cite[Question~5]{DFHM24}).
\end{enumerate}
}

\bibliographystyle{plain}
\bibliography{biblio}
\end{document}